\documentclass[11pt]{article}
\usepackage{url}
\usepackage{hyperref}
\usepackage[final]{optional}
\usepackage[a4 paper, margin=2cm]{geometry}
\usepackage{enumerate}
\usepackage{graphics}
\usepackage{tikz}
\usepackage{tikz-3dplot} 
\usepackage{verbatim}
\usepackage{slashed}
\usepackage{hyperref}

\usepackage{datetime}
\usepackage{color}
\usepackage{amsmath,amsfonts,amsthm,amssymb}
\usepackage{epstopdf}
\usepackage[all,cmtip]{xy}
\usepackage{accents}
\usepackage[nameinlink,capitalise]{cleveref} 

\newcounter{results}[section] 

\newcounter{steps}[section] 

\theoremstyle{plain}
\newtheorem{theorem}[results]{Theorem}
\newtheorem{remark}[results]{Remark}

\newtheorem{lemma}[results]{Lemma}
\newtheorem{proposition}[results]{Proposition}
\newtheorem{corollary}[results]{Corollary}

\newtheorem*{theorem*}{Theorem}
\newtheorem*{lemma*}{Lemma}
\newtheorem*{proposition*}{Proposition}
\newtheorem*{corollary*}{Corollary}
\newtheorem*{exercise*}{Exercise}
\newtheorem*{fact*}{Fact}

\newtheorem*{remark*}{Remark}
\newtheorem*{question*}{Question}

\theoremstyle{definition}
\newtheorem{definition}[results]{Definition}

\newtheorem*{definition*}{Definition}
\newtheorem*{example*}{Example}

\numberwithin{equation}{section}

\newcommand{\R}{\mathbb{R}}

\newcommand{\Sp}{\mathbb{S}}

\newcommand{\mbb}{\mathbb}
\newcommand{\C}{\mbb{C}}

\newcommand{\lan}{\langle}
\newcommand{\la}{\lambda}
\newcommand{\ra}{\rangle}

\newcommand{\emb}{\hookrightarrow}

\newcommand{\In}{\subset}
\newcommand{\Om}{\Omega}
\newcommand{\om}{\omega}
\newcommand{\dl}{{\delta}}
\newcommand{\Dl}{{\Delta}}

\newcommand{\al}{{\alpha}}

\newcommand{\ed}{{\rm d}}
\newcommand{\id}{\,\,\ed}

\newcommand{\D}{{\nabla}}

\newcommand{\ti}[1]{{\tilde{#1}}}
\newcommand{\eps}{{\varepsilon}}
\newcommand{\fr}[2]{\frac{#1}{#2}}
\newcommand{\sm}{{\setminus}}

\newcommand{\pl}[2]{{\frac{\partial #1}{\partial #2}}}
\newcommand{\ppl}[3]{{\frac{\partial^2 #1}{\partial #2 \partial #3}}}

\newcommand{\de}{\partial}
\newcommand{\nn}{\nonumber}
\newcommand{\spt}{{\rm{spt}}}

\newcommand{\Si}{\Sigma}

\newcommand{\twopartdef}[4]
{
	\left\{
		\begin{array}{ll}
			#1 & \mbox{if  } #2 \bigskip \\
			#3 & \mbox{if  } #4
		\end{array}
	\right.
}

\def\XXint#1#2#3{{\setbox0=\hbox{$#1{#2#3}{\int}$}
     \vcenter{\hbox{$#2#3$}}\kern-.5\wd0}}

\newcommand{\ov}{\overline}

\DeclareMathOperator{\Div}{\rm div}

\DeclareMathOperator{\Vol}{\rm Vol}
\DeclareMathOperator{\dist}{\rm dist}

\newcommand{\vlinesub}[1]{\vline_{_{_{_{_{_{_{_{\,#1}}}}}}}}}

\newcommand{\ppf}[4]{{\frac{{\de} #1}{{ \de} #2{ \de} #3}}\vlinesub{#4}}

\newcommand{\beq}{\begin{equation}}
\newcommand{\eeq}{\end{equation}}
\newcommand{\beqs}{\begin{equation*}}
\newcommand{\eeqs}{\end{equation*}}
\newcommand{\abs}[1]{\vert#1\vert}

\newcommand{\beqa}{\begin{equation}\begin{aligned}}
\newcommand{\eeqa}{\end{aligned}\end{equation}}
\newcommand{\beqas}{\begin{equation*}\begin{aligned}}
\newcommand{\eeqas}{\end{aligned}\end{equation*}}
\newcommand{\beqna}{\begin{eqnarray}}
\newcommand{\eeqna}{\end{eqnarray}}
\newcommand{\beqnas}{\begin{eqnarray*}}
\newcommand{\eeqnas}{\end{eqnarray*}}
\newcommand{\half}{\frac12}

\newcommand{\bl}{\pi}
\newcommand{\ZZ}{\mathcal{Z}}

\newcommand{\DD}{\mathbb{D}}

\allowdisplaybreaks

\begin{document}
\title{Low-energy 
$\al$-harmonic maps into the round sphere}
\author{Ben Sharp}
\date{\today}
\maketitle

\begin{abstract}
	We classify low-energy $\alpha$-harmonic maps from a closed non-spherical Riemannian surface $\Sigma$ of constant curvature to the round sphere via their bubble scales and centres. In particular we show that as $1<\alpha\downarrow 1$ and assuming $E_\alpha$ is close to $|
	\Sigma|+4\pi$ then degree-one $\alpha$-harmonic maps blow a bubble based at a critical point $a_c$ of a an explicit function $\mathcal{J}$ and at scale $\sqrt{
|\mathcal{J}(a_c)|^{-1}(\alpha-1)}$. Up to a constant, $\mathcal{J}$ is the sum of the squares of any $L^2$-orthonormal basis of holomorphic one-forms on the domain. 
	\end{abstract}


\section{Introduction}

One of the founding pillars of geometric analysis concerns the existence and regularity of harmonic maps. Given two Riemannian manifolds $M,N$, a harmonic map $u\in W^{1,2}(M,N)$ is a critical point (in some suitable sense) of the Dirichlet energy 
$$E(u)=\fr12 \int_M |\D u|^2 \id M.$$
When $M=\Sigma$ is two-dimensional this functional is invariant under conformal deformations of the domain which leads to a wonderfully rich theory from both an analytic and geometric viewpoint. In broad strokes this functional does not enjoy straightforward regularity properties, nor does it satisfy a Palais-Smale condition. These difficulties lead to both interesting and challenging mathematical problems. Remarkably critical points are always smooth in two dimensions thanks to H\'elein's moving frame approach \cite{helein_regularity, helein_conservation}; the resulting estimates hold only up to a potentially small regularity scale $r(u,N,\Sigma)>0$ whereupon the Dirichlet energy on any ball of radius $r$ is below some threshold. Along a sequence of smooth maps $u_k$ (for instance) it is possible for $r_k$ to shrink to zero (i.e. for the Dirichlet energy to concentrate) and all the derivatives to blow up. Suitably re-scaling the domain to reverse the energy concentration leads to a sequence of harmonic maps defined on larger and larger regions of $\R^2$ which converge to a harmonic sphere $\om:\R^2\cup\{\infty\}=\Sp^2\to N$.

Starting with the famous work of Sacks-Uhlenbeck \cite{SU81} $\alpha$-harmonic maps have been used to study existence theory for harmonic maps in two-dimensions via a regularised functional $E_\al$ and $\al>1$. Given a closed Riemannian surface $\Sigma$ and $N$ some closed Riemannian manifold, the $\alpha$-energy $E_\al$ of a map $u\in W^{1,2\al}(\Sigma, N)$ is defined as follows
$$E_\al(u):=\fr12 \int_{\Sigma} (2+|\D u|^2)^\al \id \Sigma\qquad \text{so that} \qquad E_\al(u) \xrightarrow{{\al \downarrow 1}} E(u) + |\Sigma|.$$
For $\al >1$ fixed, $E_\al$ crucially satisfies the Palais-Smale condition, whilst all finite-energy critical points ($\al$-harmonic maps) are uniformly regular, up to a regularity scale $r=r(\al,E_\al(u),N,\Sigma)$ with smooth estimates which potentially degenerate ($r\to 0$), as $\al\downarrow 1$, on regions where the Dirichlet energy concentrates. Nevertheless, it is well-known that harmonic maps appear as weak $W^{1,2}$-limits of $\al$-harmonic maps, as $\al\downarrow 1$, modulo the appearance of singularities, again taking the form of harmonic spheres (bubbles) and  geodesics of length $\ell\in[0,\infty]$ (necks) which connect the weak limit with the bubbles. An ideal outcome would be that all the Dirichlet energy is captured by the limiting harmonic map and bubbles (an energy identity) or further that all necks have zero length (a no-neck result). In general both of these ideals are known to fail during the limiting process for $\al$-harmonic maps: examples of (infinite-length) necks appearing and energy-loss can be found in \cite{LW15}. Still, one can use analytical methods to find $\al$-harmonic maps (direct minimisation, or min-max procedures) and in turn hope to recover relevant information on harmonic maps, as $\al\downarrow 1$ see e.g. the seminal works of Sacks-Uhlenbeck \cite{SU81} and Micallef-Moore \cite{MM88}.  

Natural and important questions present themselves: how precisely does this singular convergence happen and can one recover all harmonic maps from surfaces as a limit of $\al$-harmonic maps? 

When the target has positive Ricci curvature and the sequence of $\al$-harmonic maps has finite Morse index, it is known that only finite-length geodesic necks appear and all energy is accounted for in the limiting process \cite{LLW17}. For arbitrary closed targets and certain sequences of $\al$-harmonic maps obtained by min-max methods, similarly, all energy is accounted for in the limit by the work of Tobias Lamm \cite{L10}, although it is not known whether or not necks might appear. If the target is a round sphere $\Sp^n$, thanks to Jiayu Li and Xiangrong Zhu we know more: all energy is accounted for and the convergence happens in $L^\infty$ (no necks appear at all) \cite{LiZhu19}. 

When considering maps between round two-spheres $u:\Sp^2\to \Sp^2$ it is well-known that, up to reversing orientations, the only harmonic maps are holomorphic and therefore rational maps from the extended complex plane to itself (see e.g. the Theorem of Wood-Lemaire \cite[11.3]{eelm} for a more general result). Thus if we restrict to degree one harmonic maps, we are left only with the M\"obius maps and their complex conjugates. Remarkably the question ``are all harmonic maps attained as limits of $\al$-harmonic maps'' has only recently received attention. Lamm, Malchiodi and Micallef show that all degree $\pm 1$ $\al$-harmonic maps $u:\Sp^2\to \Sp^2$ with $\al$-energy lower than $8^\al 2\pi$ (Dirichlet energy lower than $12\pi$), and $\al$ close to $1$, \emph{must} be rotations \cite{LMM20,LMM21}.  Thus in particular the majority of degree-one harmonic maps between spheres are not reached by smooth limits of $\alpha$-harmonic maps. Indeed they essentially classify all low-energy $\al$-harmonic maps $u:\Sp^2\to \Sp^2$ of degree $0$ or $\pm 1$, see \cite{LMM21}.

In this article we study $\al$-harmonic maps $u:\Sigma \to \Sp^2$, from non-spherical closed surfaces $\Sigma$, which have $\al$-energy bounded above by $4\pi + |\Sigma|+\delta$ for $\dl >0$ small (in particular Dirichlet energy bounded above by $4\pi + \delta$) and $\al$ close to $1$. The energy bound immediately tells us that the topological degrees of these maps are either $0$ or $\pm 1$. 

Examples of such degree $\pm 1$ low-energy $\al$-harmonic maps are guaranteed to exist by direct methods and we show that any such map \emph{must} be quantitatively $C^1$-close, with bounds purely in terms of $\al$, to a  finite-dimensional space of singularity models $\ZZ$ where each $z\in \ZZ$ approximates a single degree $\pm 1$ bubble and constant-body map configuration. In particular we obtain estimates on both the bubble scale and bubble proximity purely in terms of $(\al-1)$ and geometric data $\mathcal{J}:\Sigma\to \R$.    

\begin{definition}\label{defSig}
	Let $(\Sigma, g)$ be a closed Riemannian surface of genus $\gamma \geq 1$ equipped with a metric $g$ of constant curvature zero when $\gamma =1$ and constant curvature $-1$ otherwise. In the flat case we also impose ${\rm Area}_g(\Sigma) =1$. Let $\{\phi_j\}$ be an arbitrary $L^2$-orthonormal basis of holomorphic one-forms on $\Sigma$ and define
	$$\mathcal{J}(a) : = -2\pi c_\gamma \sum_j |\phi_j(a)|^2_g<0$$
	where $c_1 = 1$ and when $\gamma\geq 2$, $c_\gamma = 4$. One may check directly that $\mathcal{J}$ is independent of the choice of basis.   
\end{definition}
\begin{remark}
	When $\gamma =1$ we have $\mathcal{J}\equiv -2\pi$ whilst in general we always know that $\mathcal{J} < 0$ by Riemann-Roch. $\mathcal{J}$ is directly related to the regular part of the Green's function on $\Sigma$, see e.g. Appendix \ref{App:J} and \cite{MRS23}.   
\end{remark}

Without going into fine details, we let $\mathcal{Z}\In W^{1,2}(\Sigma,\Sp^2)$ denote the space of maps approximating a single degree $\pm 1$ bubble at small scale, so topologically $\mathcal{Z}\cong \Sigma \times [\la_0,\infty)\times O(3)$ and given $a\in \Sigma$, $\la \geq  \la_0$ and $R\in O(3)$ the corresponding $z=z^R_{\la,a}\in \mathcal{Z}$ satisfies $z\sim R\pi(\la \cdot)$ \emph{near} $a$ and in suitable local coordinates,  otherwise $z\sim R(0,0,-1)$ on the rest of $\Sigma.$  Here and throughout $\pi(x)=\left(\frac{2x}{1+|x|^2} , \frac{1-|x|^2}{1+|x|^2}\right) $ is the inverse stereographic projection from the South Pole and in particular for an almost-singular limit we have $\la\gg1$, in which case $\pi (\la \cdot)$ maps a tiny disc (e.g of radius $\la^{-\fr12}$) into most of $\Sp^2\sm{(0,0,-1)}$.  

Spaces closely related to $\ZZ$ have appeared elsewhere in \cite{MRS23,R21} where they are used in the study of $H$-surfaces in $\R^3$ and harmonic maps from surfaces into arbitrary analytic targets respectively; we recall the precise definition in Section \ref{sec:bubble-space}. For now it is sufficient to note that any map which is (e.g. $W^{1,2}\cap L^\infty$) close to a single degree $\pm 1$ bubble and constant-body map configuration must be close to $\ZZ$.

In \cite{MRS23} the authors show that $\mathcal{J}(a)$ appears as the highest order term in an expansion in $\lambda$ of the $H$-energy $E^H$ when restricted to $\ZZ$. The space of singularity models $\ZZ$ appearing in \cite{MRS23} is slightly different to the one appearing here and the $H$-energy $E^H$ is the Dirichlet energy plus an enclosed volume term. In any case it is shown in that:   
$$\la^3\de_\la E^H(z) = 8\pi \mathcal{J}(a) + O(\la^{-1})\qquad\text{where $a=a(z)$ and $\la=\la(z)$.}$$
On the other hand, if $u$ is sufficiently close to $\ZZ$ and $z\in \ZZ$ is its closest-point projection (in an appropriate norm), one finds, after carefully analysing the higher-order derivatives of $E^H$ in a neighbourhood of $\ZZ$, that for $\la$ sufficiently large: 
$$\la^3 |\de_\la E^H(z)|\leq C\la (\log \la)^\fr12 \|\ed E^H(u)\|_{L^2} + \la^2\|\ed E^H(u)\|_{L^2}^2 + O(\la^{-1})(\log \la)^\fr12.$$ These estimates in particular heavily rely on a suitable non-degeneracy of the second variation in normal (to $\ZZ$) directions, along with some careful estimates on $\ed E^H$ and $\ed^2 E^H$ in a small neighbourhood of $\ZZ$.

Comparing the two facts above, one guarantees a ``quantitative non-existence result'' for almost-critical points close to $\ZZ$ in that setting, since crucially $\mathcal{J}<0$, so $\|\ed E^H (u)\|$ must be bounded away from $0$ by some suitable function of $\lambda$ - see \cite{MRS23}. As appears in the work of Rupflin \cite{R21} this expansion remains morally true (with some serious additional complications) also for the Dirichlet energy when studying almost-harmonic maps into general closed analytic targets $u:\Sigma \to N$ which are sufficiently close to a zero-body one-bubble singularity based upon any un-branched bubble $\hat\om :\Sp^2 \to N$. The situation in \cite{R21} is more subtle as the space of nearby harmonic maps (to $\hat\om$) is not restricted to pre-composition by M\"obius maps, moreover any space of singularity models $\ZZ$ must also take into account any non-integrable Jacobi fields. Nevertheless it is still the case that $\mathcal{J}(a)$ appears as the dominating term in the Dirichlet energy expansion (in $\la$), when restricted to $\ZZ$. These facts yield deep implications for the harmonic map flow from non-spherical surfaces into analytical target manifolds, and in particular on discreteness of the low-end of the Dirichlet energy spectrum in this case. On the other hand recent progress concerning low-energy discreteness of the Dirichlet energy spectrum for maps from spherical domains to analytic targets can be found in \cite{R23}. 

In this article we restrict $N=\Sp^2$ and study the behaviour of $E_\al$ in a neighbourhood of $\ZZ$ - the restriction of the target to a sphere is necessary in order to have the favourable bubble-convergence results of J. Li and X. Zhu \cite{LiZhu19}, which are unavailable (and untrue) for arbitrary targets. This then provides the technical benefit that the elements of $\ZZ$ are totally explicit. The flavour of the main results are rather different in the sense that $\al$-harmonic maps are guaranteed to exist, so our focus is very much on gleaning their properties as opposed to a quantitative non-existence result.

Following \cite{R21}, it is necessary to define the following bubble-weighted inner product and norm for $V,W\in W^{1,2}(\Sigma, \R^3)$ when $z\in \mathcal{Z}$ 
$$\lan V, W\ra_{z} =\int_\Sigma  \left(\D V \cdot \D W + \rho_z^2 V\cdot W\right) \id\Sigma, \quad \text{with} \quad  \|V\|_z = \lan V,V\ra ^\fr12$$
where, for $\la=\la(z)$ 
$$\rho_z(p) : = \twopartdef{\fr{\la}{1+\la^2{\rm d}_g(p,a)^2}}{p\in B_\iota(a)}{\fr{\lambda}{1+\la^2\iota^2}}{p\in \Sigma\sm B_\iota(a)}$$
and $\iota$ is half the injectivity radius of $\Sigma$. 

Our results are guided by the properties of  $E_\al$ when restricted to $\ZZ$. When $\la$ is bounded $E_\al(z)$ is comparable to the Dirichlet energy, so we expect to pick up similar terms in our expansion of $E_\al$ as for the Dirichlet energy, $\de_\la E_\al \sim 8\pi \mathcal{J}\la^{-3}<0$. However, for each fixed $\al>1$, it is straightforward to see that $E_\al(z)\to \infty$ as $\la\to \infty$ (indeed we have $\la^{2\al-2}\leq C(E_\al (z)+1)$ from Lemma \ref{lem:exp0}) so for large $\la$ we expect $\de_\la E_\al >0$. These facts together suggest that $\de_\la E_\al$ must vanish somewhere on $\ZZ$.  Here, we show that (Proposition \ref{prop:dela}), for $a=a(z)$ and $\la=\la(z)$, 
$$\la^{3-2\al}\de_\la E_\al(z) =8\pi\left(\mathcal{J}(a)\la^{-2} + (\al-1)\right)+O(\la^{-1}(\al-1))+ O(\la^{-3}) + O((\al-1)^2).$$ 

This expansion indicates that $\de_\la E_\al(z)$ will vanish for a $z\in \ZZ$ when $\la(z)$ and $a(z)$ approximately satisfy $\mathcal{J}(a)\la^{-2} + (\al-1) =0$. Thus for any $\al$-harmonic  map $u$ which is close to $z\in \ZZ$, we should also expect the first (highest-order) term in this expansion to be close to zero. This would force $-\mathcal{J}(a)\la^{-2}\sim (\al-1)$ and suggest the appropriate bubble scale $\la$ of $u$ in terms of $(\al-1)$. Our first result confirms this line of reasoning, and since $u$ is $\al$-harmonic we are able to also prove quantitative-closeness to the space of singularity models $\ZZ$ in $C^1$. We remind the reader that degree $\pm 1$ $\al$-harmonic maps satisfying the hypotheses of the below Theorem are guaranteed to exist by direct methods.

\begin{theorem}\label{thm:main}
Let $(\Sigma, g)$ be as in Definition \ref{defSig}. There exist $\al_0=\al_0(\Sigma)>1\,,\dl=\dl(\Sigma) >0\,,C=C(\Sigma)<\infty$ so that if $1< \al \leq \al_0$ and $u:\Sigma \to \Sp^2$ is $\al$-harmonic with $E_\al(u)\leq 4\pi + |\Sigma| + \dl$ then either 
\begin{enumerate}
	\item $u$ has degree zero and $\|\D u\|_{L^\infty(\Sigma)} \leq C$, or 
	\item $u$ has degree $\pm 1$ and there is $z(u)\in \mathcal{Z}$ realising $\inf\{\|u-z\|_z: z\in \ZZ \}$. Moreover for any such $z(u)$, setting $a=a(z(u))$ and $\la=\la(z(u))$ we have 
\begin{equation*}
	|(\al-1) + \la^{-2}\mathcal{J}(a)|\leq C(\al-1)^\fr32 |\log (\al-1)|.
\end{equation*}  
Furthermore   
\begin{equation*}
	\|u-z(u)\|_{W^{1,2\al}(\Sigma)} \leq C(\al -1)|\log(\al-1)|.\end{equation*}
	Finally, for any $s\in [1,\infty)$ there is $K=K(s, \Sigma)$ so that 
\begin{equation*}
	\|u-z(u)\|_{L^{\infty}(\Sigma)}\leq K(\al-1)^{1-\fr1s}|\log (\al-1)| \quad \text{and} \quad \|\D u -\D z(u)\|_{L^{\infty}(\Sigma)}\leq K(\al-1)^{
\fr12-\fr1s}|\log (\al-1)|.
\end{equation*}
\end{enumerate}
\end{theorem}

We furthermore study the dependence of $E_\al|_\ZZ$ on the gluing point $a$. Of course a flat torus is invariant under Euclidean translations in the domain: given $A\in T_a\Sigma$, 
$$\D_A (E_\al(z^R_{\la,a}))=\pl{}{s}\vlinesub{s=0}(E_\al (z_{\la,a_s}^R))=0\quad \text{where}\quad  \pl{a_s}{s}\vlinesub{s=0}= A,\quad a_0=a$$ in this case and there is no restriction on the location of a bubble. This is no longer true for higher-genus surfaces and we show  that (see Proposition \ref{prop:dea}):  
$$\la^{4-2\al}\D_A (E_\al(z)) = 4\pi \D_A \mathcal{J}(a) +   O(\la^{-1})+O(\la(\al-1))\quad\text{$a=a(z)$, $\la=\la(z)$.}$$	
Theorem \ref{thm:main} tells us that when $u$ is $\al$-harmonic and close to $z\in \ZZ$ then $(\al -1)\sim \la^{-2}$, so the first term above is of highest order. Once again, our rule of thumb suggests that $\D_A \mathcal{J}(a)$ must be close to zero, which we confirm next.

\begin{theorem}\label{thm:crit}
Let $(\Sigma, g)$ be as in Definition \ref{defSig} with genus $\gamma \geq 2$. There exist $\al_0=\al_0(\Sigma)>1\,,\dl=\dl(\Sigma) >0\,,C=C(\Sigma)<\infty$ so that if $1< \al \leq \al_0$ and $u:\Sigma \to \Sp^2$ is $\al$-harmonic of degree $\pm 1$ with $E_\al(u)\leq 4\pi + |\Sigma| + \dl$ then if $z(u)\in \mathcal{Z}$ achieves $\inf\{\|u-z\|_z: z\in \ZZ \}$ (as guaranteed by Theorem \ref{thm:main} part 2) then for $a=a(z)$: 
\begin{equation*}
	|\D \mathcal{J}(a)| \leq C(\al -1)^\fr12|\log(\al-1)|.
\end{equation*}  	
\end{theorem}

We immediately obtain the following compactness result: 
\begin{corollary}
Let $(\Sigma, g)$ be as in Definition \ref{defSig} and $\{u_k:\Sigma \to \Sp^2\}$ a sequence of $\al_k$-harmonic maps with $1<\al_k \downarrow 1$ and $E_{\al_k}(u_k)\to \Lambda \leq 4\pi + |\Sigma|$ then, up to subsequence, we either have 
\begin{enumerate}
	\item $u_k$ all have degree zero and they converge smoothly to a degree-zero harmonic map $u:\Sigma \to \Sp^2$ or 
	\item $u_k$ all have degree $1$ or all have degree $-1$ and they bubble-converge to a single-bubble-constant-body configuration where the bubble is blown at a critical point $a_c$ of $\mathcal{J}$ and at scale $$r_k=\sqrt{|\mathcal{J}(a_c)|^{-1}(\al_k-1)}.$$
\end{enumerate}  	
\end{corollary}

\begin{remark}
	As stated previously $\mathcal{J}$ can be defined purely via mixed second-order derivatives of the regular part of the Green's function of the Laplacian on $\Sigma$. For simply connected domains in the plane the equivalent function is known to be closely related to the Robin function - see \cite{CM05}. The critical points of the Robin function, and related functions are known to be significant in other critical variational problems see e.g. \cite{BF96,DJLW97,Han91,I101,I201,Rey90}.
\end{remark}

\paragraph{Outline of the paper}
The proofs of the main Theorems (see Section \ref{sec:proof}) rely heavily upon a careful analysis of the behaviour of $E_\al$ when restricted to the manifold $\ZZ$ (and in small neighbourhoods thereof).     
In Section \ref{sec:prelim} we recall the space of singularity models $\ZZ$ and expand the $\al$-energy along $\ZZ$ for $\lambda$ large. We recall also the crucial background results concerning $\al$-harmonic maps into round spheres \cite{LiZhu19} which we combine with suitable elliptic estimates to prove first that the $\al$-harmonic maps satisfying our hypotheses are either analytically well-controlled or lie close to $\ZZ$,  (Lemmata \ref{lem:diff} -- \ref{lem:uz}). We also recall the desired non-degeneracy of energy in normal directions which essentially follows from a similar estimate in \cite{R21} (cf \cite{MRS23}): this comparison relies heavily on the known theory for $\al$-harmonic maps into spheres, see Lemma \ref{lem:nondeg}. In Section \ref{sec:exp} we derive an expansion in $\la$ and $(\al-1)$ for $\de_\la E_\al(z)$ and $\D_A (E_\al(z))$ along the submanifold $\ZZ$. The key point of these expansions are that properties of $\mathcal{J}$ determine the behaviour of the $E_\al$-energy profile when restricted to $\ZZ$. In Section \ref{sec:est} we obtain control on $\ed E_\al$ and $\ed^2 E_\al$ when we are on and nearby $\ZZ$ - these estimates are crucially of ``lower-order'' than those appearing in Section \ref{sec:exp} which allows us to pass information on the behaviour of $E_\al$ on $\ZZ$ to the behaviour of $E_\al$ in a neighbourhood of $\ZZ$. We provide the proofs of our main Theorems in Section \ref{sec:proof} assuming the technical expansions/estimates from Sections \ref{sec:exp} and \ref{sec:est}.

\begin{center}
Acknowledgements 	
\end{center}
I would like to thank Mario Micallef and Tobias Lamm for their interest, encouragement and inspiring conversations. I would also like to thank Melanie Rupflin and Andrea Malchiodi for their helpful comments on, and interest in, the final drafts of this work.  The author was supported by the EPSRC grant EP/W026597/1.


\section{Preliminaries}\label{sec:prelim}

\subsection{Definition of the set of adapted bubbles}
\label{sec:bubble-space}

Let $\Sigma$ be as in Definition \ref{defSig}. Here we will define the set of adapted bubbles, denoted by $\ZZ=\{Rz_{\lambda,a}:\Sigma \to \Sp^2\}$ where $\lambda \gg 1$ is the bubble scale, $a\in \Sigma$ is where the bubble is attached and $R\in O(3)$. Notice this follows exactly along the lines of \cite{MRS23} and more relevantly \cite{R21} - indeed it corresponds to the same space in the latter case when $N=\Sp^2$ and $\hat\om:\Sp^2\to\Sp^2$ is the identity. We recall the full details for our setting:  

Denote by $\pi$ the inverse stereographic projection $\pi(x)=\left(\frac{2x}{1+|x|^2} , \frac{1-|x|^2}{1+|x|^2}\right) $ and we want to define $z_{\la,a}$ in a way that we have $z_{\la,a}\approx \pi_{\la}(x):=\pi(\la x)$ in   oriented isothermal coordinates $x=F_a(p)$ on the geodesic ball $B_\iota(a)\In (\Si,g)$,  $\iota:=\half \text{inj}(\Si,g)$, which are as follows: 
\begin{remark} \label{rmk:def-Fa}
If $g$ is hyperbolic we set $\rho= \tanh(\iota)$ and use that  
for any $a\in \Si$ 
there exists an orientation preserving isometric isomorphism 
$$F_a:(B_{2\iota}(a),g) \to \left(\mathbb{D}_{\rho},\tfrac{4}{(1-|x|^2)^2}g_E\right)\quad \text{with}\quad F_a(a)=0$$ where $\mathbb{D}_{\rho}:=\{x\in \R^2: \abs{x}<\rho\}$ and $g_E$ is the Euclidean metric. 
 In the flat case we will always \emph{a-priori} pick a tiling of $\R^2$ which represents $\Sigma$, set $\rho=2\iota$ and use the resulting euclidean translations $F_a$ to the origin as  coordinates.
 
 When $\gamma\geq 2$ we note that $F_a$ is determined only up to 
a Euclidean rotation on $\DD_\rho$, but that this will not affect the definition of the bubble set since this rotation can be undone by choosing an ambient rotation $R\in O(3)$ (see Remark \ref{rmk:rot} and the following definition of $\ZZ$). 
\end{remark}

In the coordinates introduced in Remark \ref{rmk:def-Fa}
\begin{equation}\label{eq:xycoords} 
F_a:B_{2\iota}(a)\times B_{2\iota}(a) \to \mathbb{D}_\rho \times \mathbb{D}_\rho, \qquad F_a(p,q) = (F_a(p), F_a(q))= (x,y)\end{equation} 
we write the Green's function $G$ via  $
G_a(x,y)= - \log |x-y| + J_a(x,y)$
for some smooth function $J$. 
 
Hence we have for $x\neq y$
\beq
\label{eq:locG-deriv}
\de_{y^i} G_a(x,y)=\frac{x^i-y^i}{\abs{x-y}^2}+ \de_{y^i} J_a(x,y). 
\eeq

 The function $J$ depends on the choice of coordinates used and is only defined when both arguments are close to $a$. We express $J_a$ in terms which clearly show dependence on the choice of coordinates via (for $p, q\in B_{2\iota}(a)$): 
 $$J^a(p,q):= G(p,q) + \log|F_a(p)-F_a(q)| \quad \text{so that} \quad J^a(p,q)=J_a(F_a(p), F_a(q))$$ and setting $\de_{q^i} = (F_a^{-1})_\ast \de_{y^i}$ we have by definition  
 $$\de_{q^i} J^a(p,q)= \de_{y^i} (J^a(p,F_a^{-1}(y))) = \de_{y^i} J_a(F_a(p),y).$$

Define $r$ via $4r=\tanh(\iota)$ in the hyperbolic case and $4r=\iota$ in the flat case. Let $\phi\in C_c^\infty(\mathbb{D}_{2r},[0,1])$ be radial with $\phi\equiv 1$ on $\mathbb{D}_r$.
Given any $a\in \Si$ and  $F_a :B_{2\iota}(a) \to \mathbb{D}_{\rho}$ a choice of isometry as in Remark \ref{rmk:def-Fa}, we define
 \begin{eqnarray}\label{eq:ztilde}
\tilde{z}_{\lambda, a}(p) = \twopartdef{\hat{z}_{\lambda,a}(F_a(p))}{p\in B_{\iota}(a)}{(\frac{2}{\lambda}(\de_{q^1}G(p,a)-\de_{q^1}J^a(a,a)), \frac{2}{\lambda}(\de_{q^2} G(p,a)-\de_{q^2}J^a(a,a)),-1)}{p\notin B_{\iota}(a)}	
\end{eqnarray}
where $\hat{z}_{\lambda, a}$ is given in the local coordinates $x=F_a(p)$ by (with $\nabla_y = (\de_{y^1}, \de_{y^2})$)
\beqa\label{eq:zhat}
\hat{z}_{\lambda,a}(x) &: = \phi(x) \left(\bl_{\lambda}(x) + \left(\tfrac{2}{\lambda}\D_y J_a(x,0)-\tfrac{2}{\lambda}\D_y J_a(0,0), 0\right)\right)\\
&\qquad  + (1-\phi(x)) \left(\tfrac{2}{\lambda} \D_y G_a(x,0)-\tfrac{2}{\lambda}\D_y J_a(0,0),-1 \right).  
\eeqa

Letting $j(x)=j_{\lambda,a}(x): = \left(\tfrac{2}{\lambda}\D_y J_a(x,0)-\tfrac{2}{\lambda}\D_y J_a(0,0), 0\right)$, $\tilde z_{a,\la}$ is described in these local coordinates on 
 $B_\iota(a)$ by
 \beqa
\hat{z}_{\lambda,a} (x)&= \bl_{\lambda }(x)  + j(x)   +(\phi(x)-1)\left(\left(-\tfrac{2}{\lambda} \D_y G_a(x,0)+\tfrac{2}{\lambda}\D_y J_a(x,0), 1\right) + \bl_{\lambda }(x)  \right) \\
&= \bl_{\lambda }(x) + j(x) + (\phi(x)-1)O(\lambda^{-2}).\label{eq:zloc}\eeqa
The error in \eqref{eq:zloc} follows from the definition of $\pi_\la$ and \eqref{eq:locG-deriv}: 
$$(-\tfrac{2}{\lambda} \D_y G_a(x,0)+\tfrac{2}{\lambda}\D_y J_a(x,0), 1) + \bl_{\lambda }(x)=\fr{2}{1+\la^2|x|^2}\left(-\frac{x}{\la|x|^2}, 1\right)=O(\la^{-2}) \quad \text{on $\DD_{2r}\sm \DD_r$.}$$
\begin{remark}\label{rmk:jharm}
The conformal covariance of the Laplacian implies that $\de_{y^i} J_a(x,0)$  is harmonic in $x$, so in particular $j$ is also harmonic.

On an unrelated note, for symmetry reasons it is possible to check that $\D_y J_a(0,0) = 0$ when $\gamma=1$, however it is unclear whether this remains true for $\gamma\geq 2$. 
\end{remark}
\begin{remark}\label{rmk:rot}
When $\gamma\geq 2$ we note the dependence on the choice of $F_a$ via (for $p$ close to $a$):  
\beqas\tilde{z}^{F_a}(p) = \tilde{z}_{\la,a}(p)= \pi_\la (F_a(p)) + j^{F_a}(p)  +(\phi(F_a(p))-1)O(\la^{-2}),\eeqas
with $j^{F_a}(p)=j_{\la}^a(p) =\fr{2}{\la}\left(\de_{q^1} J^a(p,a)-\de_{q^1}J^a(a,a), \de_{q^2}J^a(p,a)-\de_{q^2}J^a(a,a), 0\right)$. Thus if we pick another isometry as per Remark \ref{rmk:def-Fa} $\check{F}_a$ we have $\check{R}=F_a\circ \check{F}_a^{-1}\in SO(2)$ and now let $R$ be the obvious extension of $\check{R}$ to $O(3)$ (as the identity in the third component). Following the definitions above we have  
$$\tilde{z}^{F_a}(p)=R\tilde{z}^{\check{F}_a}(p).$$
Thus a different choice of $F_a$ in Remark \ref{rmk:def-Fa}  has the effect of an ambient rotation and does not change the eventual bubble set $\ZZ$ defined below. 	
\end{remark}

We then set
\beqs
z_{\lambda,a} = P(\tilde{z}_{\la,a}):=\frac{\tilde{z}_{\la,a}}{|\tilde{z}_{\la,a}|} \quad \text{and define} \quad \mathcal{Z}:=\{z=z^R_{\la,a}=R z_{\lambda,a} | a\in \Sigma, R\in O(3), \lambda > 1\}.\eeqs
See Appendix \ref{app:A} for further expansions and estimates for $z\in\ZZ$ and their derivatives. Here are throughout we will also use the notation 
\begin{equation*}
	\ZZ^{\al,\dl}:= \{z\in \ZZ :  E_\al(z)-|\Sigma| - 4\pi \leq \dl\}.
\end{equation*}

\begin{remark}[\textbf{Notation}]\label{rmk:not}
The notation $f=O(g)$ means that there exists some $C<\infty$ so that $|f|\leq Cg$ and $f\simeq g$ means there is $C<\infty$ so that $C^{-1}|g|\leq |f| \leq C|g|$. The constant $C$ may change from one line to the next but in any given statement $C$ is fixed and depends at most upon the injectivity radius of, and the Green's function on, $\Sigma$ (sometimes neither). In particular for any fixed $\Sigma$ as in Definition \ref{defSig}, $\emph{a-posteriori}$ there is some uniform $C$ which works for all statements in this paper.  

	We will use the notation $\Dl$, $\D$, $\cdot$, $\id \Sigma$ to denote these objects with respect to the metric $g$ on $\Sigma$ \textbf{unless} we are in $x=F_a(p)$ coordinates at which point we will use the notation $\D^g$, $\Dl_g$, $\cdot_g$, $\ed x_g$ etc to denote these objects. In these coordinate only, the same symbols without a $g$ are the standard Euclidean ones (as in e.g. \eqref{eq:zhat} above). \end{remark}

In the first instance we will determine the behaviour of $E_\al$ on $\ZZ$ - whereas in Sections \ref{sec:exp} and \ref{sec:est} we will more carefully analyse the behaviour of $\ed E_\al$ and $\ed^2 E_\al$ on and close to $\ZZ$.

\begin{lemma}\label{lem:exp0}
Let $\al\leq 2$, $\dl \leq 1$ and $z\in \mathcal{Z}^{\al,\dl}$, then for $\la=\la(z)$,
\begin{equation*}
E_\al (z) -|\Sigma| - 4\pi = (\la^{2\al-2}-1)4\pi +O(\la^{2\al-4}) + O((1+ \la^{2\al-2})(\al-1)).
\end{equation*}
	
\end{lemma}
\begin{corollary}\label{cor:}
	Let $1<\al\leq 2$, $\dl \leq 1$ and $z\in \mathcal{Z}^{\al,\dl}$, then there exists $C=C(\Sigma)<\infty$ so that 
$\la(z)^{2\al-2} \leq C$. Furthermore for any $\eps>0$ there is $\al_0=\al_0(\Sigma)>1$ and $\dl=\dl(\Sigma) >0$ so that $\al\leq \al_0$ and $z\in \ZZ^{\al,\dl}$ implies $$|\la(z)^{2\al-2} - 1|<\eps .$$

\end{corollary}
%

\begin{proof}
\begin{eqnarray*}
E_\al (z)&=& 2^{
\al-1}|\Sigma|+ \fr12 \int_{\Sigma} \Big( (2+|\D z|^2)^\al - 2^\al \Big) \id\Sigma\\	
&=& 2^{\al-1}|\Sigma| + O(\la^{-2}) + \fr12\int_{\DD_r}\Big((2+|\D^g z|_g^2)^\al - 2^\al\Big) \id x_g
\end{eqnarray*}
where $g$ denotes the metric in $x$-coordinates. With $c_\gamma\in\{1,4\}$ defined as in Definition \ref{defSig} we have $\id x_g = (c_\gamma+O(|x|^2)) \id x$ and $|\D^g z|_g^2=(c_\gamma^{-1} +O(|x|^2))|\D z|^2$, noting that when $\gamma =1$ we have $c_\gamma =1$ and no error (i.e. everything is Euclidean).   

On $\DD_r$ we have from \eqref{eq:zr} and \eqref{eq:rough1}, 
\beqna
\D z &=& \D \pi_\la  + \D j_{\lambda,a} ^\top + \D K_{\lambda,a} \nn \\
&=& \D \pi_\la + O(\la^{-1}). \label{eq:dz}
\eeqna
 
Using that, for $x\in \DD_\rho$, $O(\la^{-2}) + O(\la^{-1})\D \pi_\la = O(|\D \pi_\la|^2 (\la^{-2} + |x|^2))$ we obtain 
\beq\label{eq:dz2}
|\D^g z|_g^2 =|\D \pi_\la|^2 (c_\gamma^{-1}+ \chi(\la,|x|)) 
\eeq where $\chi(\la,|x|)=O(\la^{-2} +  |x|^2)$. 

By the mean value formula, we note that (given $a,b >0$) there is some $c\in(0, a)$ so that  
\beq\label{eq:mv}
(a+b)^{\al-1} - b^{\al-1} = (\al-1)a(c+b)^{\al-2}  \leq (\al-1)ab^{\al-2}.
\eeq
Employing this twice, with the roles of $a$ and $b$ are reversed, we may write 
$$(2+|\D^g z|_g^2)^\al = (a+b)^\al= a(a+b)^{\al-1} + b(a+b)^{\al-1}=a^\al + b^\al + (\al-1)e_0 $$
where $a=2$ and $b=|\D \pi_\la|^2 (c_\gamma^{-1}+\chi)$ and $e_0 \leq (ab^{\al -1} + a^{\al-1}b)$.

This yields 
\begin{eqnarray*}
\left((2+|\D^g z|_g^2)^\al - 2^\al\right) \id x_g = |\D \pi_\la|^{2\al} (c_\gamma^{-1}+\chi)^\al \id x_g +(\al-1)(O(|\D \pi_\la|^{2\al-2}) + O(|\D \pi_\la|^2)) \id x.	
\end{eqnarray*}
Using that $(c_\gamma^{-1}+\chi)^\al = c_\gamma^{-\al}+ \chi$ for $\chi = O(\la^{-2} + |x|^2)$ we now have 
\begin{eqnarray*}
E_\al (z)	
&=& 2^{\al-1}|\Sigma| + O(\la^{-2}) \\
&& + \fr{1}{2}\int_{\DD_r} c_\gamma^{1-\al}|\D \pi_\la|^{2\al} + \chi|\D \pi_\la|^{2\al} + (\al-1)(O(|\D \pi_\la|^{2\al-2}) + O(|\D \pi_\la|^2))\id x.
\end{eqnarray*} 
From e.g. \eqref{eq:pop} we have 
\begin{eqnarray*}
\int_{\DD_r}|\D \pi_\la|^{2\al} = \pi\la^{2\al-2}\frac{8^{\al}}{2\al-1} + O(\la^{-2\al}) &\qquad & \int_{\DD_r} |\D \pi_\la|^{2\al-2} = O(\la^{2\al-4})\\
 \int_{\DD_r}|\D \pi_\la|^2 =	O(1) &\qquad & \int_{\DD_r}\chi|\D \pi_\la|^{2\al}= O(\la^{2\al -4})
\end{eqnarray*}
which gives  
\begin{equation*}
E_\al (z) - 2^{
\al-1}|\Sigma| = 4\pi \la^{2\al-2}\frac{8^{\al-1}c_\gamma^{1-\al}}{2\al-1} +O(\la^{2\al-4}) + O((\al-1))
\end{equation*}
from which the result follows.  
	
\end{proof}

\subsection{Analytical set up and further background results}

We crucially require the energy identity and no-neck result of Jiayu Li and Xiangrong Zhu for $\al$-harmonic maps into round spheres \cite[Theorem 1.1, Lemma 4.1]{LiZhu19}. In particular their results straightforwardly  imply: 
\begin{theorem}{\cite[Theorem 1.1, Lemma 4.1]{LiZhu19}}\label{thm:bub}
	\\
	If $\{u_k\} \in W^{1,2\al_k}(\Sigma, \Sp^2)$ is a sequence of $\al_k$-harmonic maps with $1<\al_k\downarrow 1$ and $E_{\al_k}(u_k) \leq \Lambda$ then there exist a finite number of bounded-energy smooth harmonic maps $u_\infty:\Sigma \to \Sp^2$, $\{\psi^j\} :\R^2 \to \Sp^2$ and accompanying distinct point-scale sequences $(a^j_k,r^j_k)$ so that $a^j_k \to a_k \in \Sigma$, $r^j_k\to 0$ and: 
\begin{enumerate}
	\item For all $j$ we have $(r_k^j)^{1-\al_k}\to 1$, and  $\|(2+|\D u_k|^2)^{\al_k-1}\|_{L^{\infty}(\Sigma)}\to 1$
	\item $E_{\al_k}(u_k) \to |\Sigma| + E(u_\infty) + \sum_j E(\psi^j)$
	\item $$\left\| u_k -  u_\infty - \sum_j  \phi^j_k(\cdot) \left( \psi^j((r_k^j)^{-1}F_{a^j_k}(\cdot)) - 
\psi^j(\infty) \right) \right\|_{W^{1,2\al}(\Sigma)} \to 0,$$ and $$\left\| u_k -  u_\infty - \sum_j  \phi^j_k(\cdot) \left( \psi^j((r_k^j)^{-1}F_{a^j_k}(\cdot)) - 
\psi^j(\infty) \right) \right\|_{L^{\infty}(\Sigma)} \to 0,$$ 

where $\phi^j_k \in C^\infty(\Sigma)$ is defined by $\phi^j_k (p) =\phi(F_{a^j_k}(p))$ on $\spt(\phi^j_k)\In B_\iota (a^j_k)$ for some fixed $\phi\in C_c^\infty(\mathbb{D}_{\rho},[0,1])$ with $\phi\equiv 1$ on $\mathbb{D}_{\rho/2}$ and $\iota$, $\rho$ are defined as in Remark \ref{rmk:def-Fa}. 

\end{enumerate}
	
\end{theorem}
\begin{remark}\label{rmk:onebub}
Notice that in the scenario that one bubble forms and $u_\infty$ is a constant, we must have $\psi(\infty) = u_\infty$ and in the above convergence we may write 
$$\|u_k - \phi_k(\cdot)(\psi(r_k^{-1}F_{a_k}(\cdot)) -(1-\phi_k(\cdot))\psi(\infty)\| \to 0$$
for both norms.  	
\end{remark}

\paragraph{Weighted inner product:}
For $v\in C^{\infty}(\Sigma,\Sp^2)$ we denote by 
$$\Gamma^2(v):=\{V\in W^{1,2}(\Sigma,\R^3): V(x)\in T_{v(x)}\Sp^2, \, {\rm a.e.}\,\, x\in \Sigma\}.$$ Following \cite{R21}, it is necessary to define the following inner product for $V,W\in W^{1,2}(\Sigma, \R^3)$ when $z\in \mathcal{Z}$ (so in particular also for $V,W\in\Gamma^2(z)$)   
$$\lan V, W\ra_{z} =\int_\Sigma  \left(\D V \cdot \D W + \rho_z^2 V\cdot W\right) \id\Sigma, \quad \text{with} \quad  \|V\|_z = \lan V,V\ra ^\fr12$$
where 
$$\rho_z(p) : = \twopartdef{\fr{\la}{1+\la^2{\rm d}_g(p,a)^2}\simeq|\D z|}{p\in B_\iota(a)}{\fr{\lambda}{1+\la^2\iota^2}}{p\in \Sigma\sm B_\iota(a)}$$
so in particular $|\D z|\leq C\rho_z$ on $\Sigma$. 

\begin{remark}\label{rmk:2.28}
	We have the following useful estimate from \cite[(2.28) and Appendix B]{R21}:
	For any $s\in [1,\infty)$ there exists $K=K(s,\Sigma)$ so that 
	\begin{equation*}
		\left|\int_{\de B_\iota(a)} v \id \Sigma\right| + \|v\|_{L^s(\Sigma)}\leq K(s,\Sigma)(\log \la)^{\fr12}\|v\|_z.
	\end{equation*}
\end{remark}
\begin{lemma}\label{lem:diff}
If $u:\Sigma \to \Sp^2$ is an $\al$-harmonic map with $\|u-z\|^2_z<\tfrac{\eps_0}{4}$ (with $\eps_0$ from Theorem \ref{thm:reg}), $\al >1$ and and $z\in \ZZ$ then there is some $C=C(\Sigma)<\infty$ so that $\|\D u\|_{L^{\infty}}\leq C\la(z)$ and we have: for any $p\in \Sigma$ and $s\in [1,\infty)$ there exists $K(s,\Sigma)<\infty$  
	$$|u(p)-z(p)|+ \la^{-1}|\D u(p) - \D z(p)| \leq K(s,\Sigma)(\la^{\fr2s}(\log \la)^{\fr12}\|u-z\|_z + (\al-1) + \la^{-2}).$$  	
\end{lemma}

\begin{proof}
For $z\in \ZZ$ with $\|u-z\|^2_z<\tfrac\eps4$ let $\la$ and $a$ be defined by $z$. 

Fix $\mu=\tfrac12\iota$ and notice first that for any $p$ with ${\rm d}_\Sigma(p,a)> \mu$ there is some uniform radius $r>0$ so that $\|z\|_{C^2(B_r(p))}\leq C\la^{-2}$. Thus, there is some uniform radius $r$ so that $\|\D u\|^2_{L^2(B_r(p))}\leq 2\|u-z\|^2_z + Cr^2\la^{-4}<\eps_0$ ($\eps_0$ from Coroallry \ref{cor:reg}) meaning that $\|\D u\|_{L^{\infty}(B_{r/2})}\leq C( \|u-z\|_z + \la^{-2})$ by the regularity of $\al$-harmonic maps. Thus the estimates trivially hold on $\Sigma \sm B_{\frac\iota2}(a)$.

On $B_\iota(a)$ our usual coordinates $x$ are defined on $\DD_R$ where $R=\tanh(\iota/2)$ when $\gamma \geq 2$ else $R=\iota$. Let $\hat{z}(x) = z(\la^{-1} x)$, which satisfies $\|\hat{z}\|_{C^2(\DD_{\la R})}\leq C$ for some uniform constant $C$. In particular there is some uniform radius $r$ so that 
$$\int_{\DD_r(x)}|\D \hat z|^2 <\frac{\eps_0}{4}$$
for any $\DD_r(x)\In \DD_{\la R}$. Defining $\hat{u}(x) = u(\la^{-1}(x))$ in these coordinates, and the scale-invariance of $\|\cdot\|_z$ give that 
$$\int_{\DD_r(x)}|\D \hat u|^2 \leq 2\|u-z\|_z^2 + 2\int_{\DD_r(x)} |\D \hat{z}|^2<\eps_0$$
and since $u$ is $\al$-harmonic $\hat{u}$ solves, defining $f(x) = \frac{4}{(1-|x|^2)^2}$ when $\gamma\geq 2$ else $f\equiv 1$, and $\hat{f}(x) = f(\la^{-1}x)$: 
$$-\Dl \hat{u} = \hat{u}|\D \hat{u}|^2 + (\al-1)\frac{\D(\hat{f}^{-1}|\D\hat{u}|^2)\cdot \D \hat{u}}{\la^{-2}+ \hat{f}^{-1}|\D \hat{u}|^2}$$
meaning that $\D \hat{u}$ satisfies uniform $L^\infty$-estimates at this scale by Theorem \ref{thm:reg}. In particular defining $\mu = \frac{\tanh(\tfrac\iota4)}{\tanh(\tfrac\iota2)}$ when $\gamma\geq 2$ or $\mu=\tfrac34$ otherwise, $\|\hat{u}\|_{C^1(\DD_{\mu\la R})}\leq C$ and by undoing the scaling $\|\D u\|_{L^\infty(\DD_{\mu R})}\leq C\la$ and $\|\D u\|_{L^\infty(B_{\frac\iota2})} \leq C\la$.

Using Theorem \ref{thm:reg} again tells us that 
$$-\Dl \hat{u} = \hat{u}|\D \hat{u}|^2 + (\al-1)g \quad \text{where} \quad \|g\|_{L^4(\DD_2(x))}\leq C$$
for any disc of radius $2$ in $\DD_{\la R}$. 

Now, in our coordinates we have that $-\Dl \hat{z} = \hat{z}|\D \hat{z}|^2 + O(\la^{-2})$ on $\DD_{\la R}$, as can be checked directly from the definition. Thus
$$-\Dl (\hat{u} - \hat{z})=(\hat{u}-\hat{z})|\D \hat{z}|^2 + \hat{u}(\D \hat{u} - \D \hat{z})\cdot (\D \hat{u} + \D \hat{z}) + O(\la^{-2}) + (\al-1)g.$$
Standard elliptic estimates, using the scale-invariance of the $\|\cdot \|_z$-norm, first give for all $s\geq 1$ (on balls with arbitrary centre $x\in \DD_{\mu \la R}$) the existence of a constant $K_s$  so that  
$$\|\hat{u}-\hat{z}\|_{W^{2,2}(\DD_{\fr32}(x))}\leq K_s(\|u-z\|_z +\la^{-2} + (\al-1) + \|\hat{u}-\hat{z}\|_{L^s(\DD_2(x))})$$ and then by Sobolev embedding  
$$\|\hat{u}-\hat{z}\|_{W^{2,4}(\DD_{1}(x))}\leq K_s(\|u-z\|_z +\la^{-2} + (\al-1) + \|\hat{u}-\hat{z}\|_{L^s(\DD_2(x))})$$
	By undoing the scaling and using Remark \ref{rmk:2.28}
	we have 
	$$\|\hat{u}-\hat{z}\|_{L^s}\leq \la^{\fr2s}\|u-z\|_{L^p(\Sigma)}\leq K(s,\Sigma)\la^{\fr2s}(\log \la)^{\fr12}\|u-z\|_z.$$
We are left with, in particular   
$$\|\hat{u}-\hat{z}\|_{L^\infty (\DD_{\mu \la R})} + \|\D\hat{u} - \D \hat{z}\|_{L^\infty (\DD_{\mu \la R})}\leq K(s,\Sigma)(\la^{\fr2s}(\log \la)^{\fr12}\|u-z\|_z + (\al-1) + \la^{-2})$$
and therefore 
$$\|u-z\|_{L^\infty(B_{\iota/2}(a))}+\la^{-1}\|\D u - \D z\|_{L^\infty(B_{\iota/2}(a))} \leq K(p,\Sigma)(\la^{\fr2p}(\log \la)^{\fr12}\|u-z\|_z + (\al-1) + \la^{-2}).$$
\end{proof}

\begin{lemma}\label{lem:0}
There exist $\al_0=\al_0(\Sigma)>1$, $\dl=\dl(\Sigma)>0$ and $C=C(\Sigma)<\infty$ so that if $1< \al \leq \al_0$, and $u:\Sigma \to \Sp^2$ is a degree zero $\al$-harmonic map with $E_\al(u) \leq 4\pi + |\Sigma| + \dl$ then $\|\D u\|_{L^{\infty}(\Sigma)} \leq C$.  
\end{lemma}
\begin{proof}
Suppose not, then there is a sequence $u_k$ of $\al_k$ harmonic maps of degree zero with $\al_k\downarrow 1$, $E_\al(u_k)\to 4\pi + |\Sigma|$ and $u_k$ bubble converges to a limit with at least one bubble. In fact it must have exactly one bubble of degree $\pm 1$ and energy $4\pi$ so $u_\infty$ is a constant map. Since bubble convergence preserves homotopy type in this case  we have contradicted the fact that each $u_k$ has degree zero.  	
\end{proof}

Using the above, along with a direct application of Lemma 4.2 in \cite{R21} we are able to prove  
\begin{lemma}\label{lem:uz}
For all $\eps_1 >0$ there exist $\al_0=\al_0(\Sigma) >1$, $\dl=\dl(\Sigma)>0$ so that if $1<\al\leq \al_0$ and $u:\Sigma \to \Sp^2$ is a degree-one $\al$-harmonic map with $E_\al(u) \leq 4\pi + |\Sigma| + \dl$, then 
\begin{enumerate}
\item $\|\D u\|_{L^{\infty}} >\eps_1^{-1}$ 
	\item $1\leq \|(2+|\D u|^2)^{\al-1}\|_{L^{\infty}(\Sigma)} <1+\eps_1$
	\item There exists $z\in \ZZ$ so that $\|u-z\|_{W^{1,2\al}(\Sigma)} + \|u-z\|_{L^{\infty}(\Sigma)} <\eps_1$.
	\item (from \cite[Lemma 4.2]{R21}) There exists $z\in \ZZ$ so that $\|u-z\|_z = \inf_{\zeta\in \ZZ}\{\|u-\zeta\|_\zeta\}$ and for this $z$ we may still assume that $\|u-z\|_{L^{\infty}}<\eps_1$
	\end{enumerate}

\end{lemma}

\begin{proof}
Suppose not, then in particular we know that there is some $\eps_1>0$ and sequences $\al_k \downarrow 1$, $\dl_k\to 0$ and maps $u_k$ as in the statement for which one of $1-3$ fails for infinitely many $k$ and for this $\eps_1>0$.

$2.$ cannot fail for infinitely many $k$ since either $\|\D u_k\|$ remains uniformly bounded (in which case it is trivially true) or $u_k$ converges as in Theorem \ref{thm:bub}, for which we have that $2.$ holds eventually.  

If $1.$ fails for infinitely many $k$ then the regularity theory of $\al$-harmonic maps implies that $u_k$ must converge smoothly  to a harmonic map $u_\infty : \Sigma \to \Sp^2$ which itself has degree one and $E(u_\infty) = E_1(u_\infty)-|\Sigma|=4\pi$. Since this map has degree one we know that the area of its image is at least $4\pi$, thus its Dirichlet energy equals its area and it must be conformal. We have proved the existence of a degree one holomorphic map from $\Sigma$ to $\Sp^2$, contradicting the fact that $\Sigma$ is not a sphere. 

Thus in particular we cannot have smooth convergence to a limit, and we must have convergence to a non-trivial bubble configuration of total energy $4\pi$ and degree one. Since all harmonic spheres are conformal and therefore rational maps and their complex conjugates there is exactly one bubble $\psi$ of degree one (or minus one) and $u_\infty$ must be a constant which we write in the form $R(0,0,-1)$ for some appropriate $R\in O(3)$. Thus $\psi$ is of the form $R\pi_\la (x-x_0)$ for some fixed $x_0$. By a translation in the domain and rotation in the target we may assume that $\psi = \pi_{\la_0}$ and we have that $u_k$ converges as in Theorem \ref{thm:bub}. Letting $z_k$ be defined via $\la = r_k^{-1}\la_0 \to \infty$ we see immediately that $3.$ cannot fail for infinitely many $k$ (see also Remark \ref{rmk:onebub}). 

We have proven the required result for the first three points. 

For $4.$: it follows directly from Lemma 4.2 in \cite{R21} that for any $\nu_1 >0$ there is $\nu_2>0$ so that if there is $z_0\in \ZZ$ with  
$$\|u-z_0\|_{L^{\infty}} + \|u-z_0\|_{W^{1,2}} \leq \nu_2$$
then there is $z\in \ZZ$ realising $\dist_z(u,\ZZ)\leq C\nu_2$ so that 
$$\|u-z\|_{L^{\infty}}\leq \nu_1,$$
thus by picking $\al_0-1$ and $\dl$ sufficiently small we can ensure $4.$ holds also.  
\end{proof}

It is crucial that the energy is non-degenerate in normal directions (to $\ZZ$) in our setting. The proof of this fact  follows from the analogous results for the Dirichlet energy \cite{R21} (cf \cite{MRS23} for $H$-surfaces) - see the claim and its proof in the below.   
\begin{lemma}\label{lem:nondeg}
There exist $\al_0=\al_0(\Sigma)>1$, $\dl=\dl(\Sigma)>0$, $c_0>0$ so that for any $\al\leq \al_0$ and $u:\Sigma\to \Sp^2$ $\al$-harmonic, degree one, $E_\al(u) - 4\pi - |\Sigma| < \dl$, and if $z\in \ZZ$ realises $\|u-z\|_z = \inf_{\zeta\in \ZZ}\{\|u-\zeta\|_\zeta\}$, then setting $W=\ed P(z)[u-z]$ and $w=u-z$ we have 
$$\ed^2 E_\al(z)[W,W]\geq c_0 \|w\|_z^2.$$

\end{lemma}

\begin{proof}
Write $T_z \ZZ ^{\bot_z}$ to mean the $\lan,\ra_z$-orthogonal complement of $T_z\ZZ\In W^{1,2}(\Sigma,\R^3)$. We first    
\paragraph{Claim:} If $V\in \Gamma^2(z)\cap T_z\ZZ^{\bot_z}$ then there is $\ti{c}_0>0$ so that 
\beqs
\ed^2 E_\al(z)[V,V] \geq \ti{c}_0 \|V\|_{z}^2.   
\eeqs

Assuming the Claim, it follows from \cite[Lemma 4.3]{R21}, that if $Y=W^{\top_{T_z \ZZ}}\in \Gamma^2(z)\cap T_z Z$ and  $V=W^{\bot_{T_z \ZZ}}\in \Gamma^2(z)\cap T_z\ZZ^{\bot_z}$ then there exists $C<\infty$ so that
$$\|Y\|_z^2\leq C\|w\|^2_{L^\infty}\|w\|^2_z$$ and therefore  $$\|V\|^2_z \geq \|W\|_z^2 - C\|w\|^2_{L^\infty}\|w\|^2_z.$$ 

Since $W=\ed P(z)[u-z]$ we have (by checking directly, or see \eqref{eq:Ww}) $\|W\|_z \geq \fr12 \|w\|_z$ so that the above becomes, after choosing $\al_0$, $\dl$ appropriately and applying Lemma \ref{lem:uz} 
\beqs\label{eq:V}
\|V\|^2_z \geq \fr14\|w\|_z^2.
\eeqs

For any $\eps>0$, once again ensuring that $\al_0$, $\dl$ are sufficiently small, we may ensure from Lemmata \ref{lem:diff} and \ref{lem:uz} $\la(z)>C\eps^{-1}$ and from Corollary \ref{cor:} that 
\beq\label{eq:small}
1\leq (2+|\D z|^2)^{\al-1} \leq 1+ C\eps 
\eeq
which in turn, amongst other things yields that 
\beqnas
\ed^2 E_\al(z)[A,B] &=& \int_\Sigma (2+|\D z|^2)^{\al-1} (\D A \cdot \D B - |\D z|^2 A\cdot B) \nn\\
&&\quad +2\al(\al-1)\int_{\Sigma}  (2+|\D z|^2)^{\al-1} \frac{(\D z \cdot \D A)(\D z \cdot \D B)}{2+|\D z|^2}\nn \\
&\leq &C\|A\|_z\|B\|_z. 
\eeqnas 
Putting the above threads together, along with some straightforward application of Young's inequality with a $\nu$ and possibly after lowering $\al_0$ and $\dl$ we have some $c_0>0$ so that 
\beqas
\ed^2 E_\al(z)[W,W] &= \ed^2 E_\al(z)[V,V] + \ed^2 E_\al(z)[Y,Y] +2 \ed^2E_\al(z)[V,Y] \\
&\geq \ti{c}_0\|V\|_z^2 - C_\nu\|Y\|_z^2 - \nu \|V\|_z^2  \\
&\geq c_0\|w\|_z^2   
\eeqas  
as required. 

\

It remains to prove the \textbf{Claim}. From \eqref{eq:small}
\beqnas
\ed^2 E_\al(z)[V,V] &=& \int_\Sigma (2+|\D z|^2)^{\al-1} (|\D V|^2 - |\D z|^2|V|^2) \nn\\
&&\quad +2\al(\al-1)\int_{\Sigma}  (2+|\D z|^2)^{\al-1} \frac{(\D z \cdot \D V)^2}{2+|\D z|^2}\nn\\
&\geq &  \int_\Sigma  |\D V|^2 - |\D z|^2|V|^2 - C\eps \|V\|_z^2 \nn \\
&=&  \ed^2 E(z)[V,V] - C\eps \|V\|_z^2. 
\eeqnas
We now appeal directly to Lemma 3.2 in \cite{R21} which tells us that all eigenvalues of $\ed^2 E(z)$ are uniformly bounded away from zero on $(T_z \ZZ)^\bot_z$, for $\la$ sufficiently large, which for us is ensured by Lemmata \ref{lem:diff} and \ref{lem:uz} implying $\la(z)>C\eps^{-1}$. In our case all eigenvalues must be non-negative since the limiting bubble (the identity map on the sphere) is stable, so all eigenvalues are uniformly positive and bounded away from zero: there exists $\hat{c}_0$ so that  
$$\ed^2 E(z)[V,V] \geq \hat{c}_0\|V\|_z^2.$$ By potentially lowering $\al_0$, $\dl$ to make $\eps$ sufficiently small, we have proved the claim.

\end{proof}

\section{Proof of the main theorems assuming further technical details}\label{sec:proof}


\begin{proof}[Proof of Theorem \ref{thm:main}]

Let $u$ be an $\al$-harmonic map as per the hypotheses. For energy reasons $u$ must have degree zero or one. If it has degree zero then Lemma \ref{lem:0} proves the first part of the Theorem.

If $u$ has degree one, we may choose $\dl$, $\al_0$ as in Lemma \ref{lem:uz} associated with some $\eps < \tfrac{\eps_0}{4}$ and let $z$ be a resulting element in $\ZZ$. Setting $w=u-z$, $u_t = P(z+t(u-z))$ where $P(q)=\frac{q}{|q|}$ is the closest point projection to $\Sp^2$. We let $W_t=\de_t u_t \in \Gamma^2(u_t)$ with $W_0=W=\ed P(z)[w]$.

Then 
\begin{eqnarray*}
	0=\ed E_\al (u)[W_1]&=& \int_0^1 \pl{}{t} \ed E_\al (u_t)[W_t]\id t + \ed E_\al (z)[W] \\
	&=&  \int_0^1 \ed^2 E_\al(u_t)[W_t, W_t] - \ed^2E_\al (z)[W,W]+  \ed E_\al(u_t)[\ed P(u_t)\de_t W_t]   \id t \\
	&&+\ed^2 E_\al (z)[W, W] + \ed E_\al (z)[W].
\end{eqnarray*} 	
Now by potentially lowering $\al_0$, $\dl$ and utilising Lemma \ref{lem:nondeg} there is some $c_0>0$ so that  
$$\ed^2E_\al (z)[W,W] \geq c_0 \|w\|_z^2.$$
By Lemmata \ref{lem:e3} and \ref{lem:uz} along with Corollary \ref{cor:} we may lower $\al_0$, $\dl$ sufficiently so that 
$$|\ed^2 E_\al(u_t)[W_t, W_t] - \ed^2E_\al (z)[W,W] | + |\ed E_\al(u_t)[\ed P(u_t)\de_t W_t]|\leq \frac{c_0}{2}\|w\|_z^2.$$
Lemma \ref{lem:V} combined with \eqref{eq:Ww} gives   
$$|\ed E_\al(z)[W]|\leq C(\log \la)^\fr12(\la^{-2} + (\al-1))\|w\|_z.$$

Putting the above estimates together gives 
\begin{equation}\label{eq:w0}
	\|w\|_z \leq C(\log \la)^\fr12(\la^{-2} + 
(\al-1)).
\end{equation}

Now, for $T\in T_z \mathcal{Z}$ we set $T_t = \ed P(u_t)[T]$ and note that  
\begin{eqnarray}\label{eq:main}
	0=\ed E_\al(u)[T_1] &=& \int_0^1 \pl{}{t}\ed E_\al (u_t)[T_t] \id t + \ed E_\al(z)[T]\nn \\
	&=&\int_0^1 \ed^2 E_\al(u_t)[T_t, W_t] -  \ed^2 E_\al(z)[T,W] + \ed E_\al(u_t)[\ed P(u_t)\de_t T_t] \nn\\
	&& + \ed^2 E_\al(z)[T,W] + \ed E_\al(z)[T].
\end{eqnarray}

We will apply this formula for $T^\la=\de_\la z$ which gives, first using Lemma \ref{lem:V},  \eqref{eq:Ww} and then \eqref{eq:w0}: 
$$|\ed^2 E_\al(z)[T^\la,W]|\leq C(\log\la)^{\fr12}(\la^{-2} + \la^{-1}(\al-1))\|w\|_z\leq C\la^{-1}(\log \la)(\la^{-3} + \la^{-1}(\al-1) + (\al-1)^2)$$
and Lemma \ref{lem:e4} with \eqref{eq:w0}:  
\beqas
	|\ed^2E_\al(z)[T^\la,W]-\ed^2 E_\al(u_t)[T^\la_t,W_t]& + \ed E_\al (u_t)[\ed P(u_t)\de_t T^\la_t]|\leq C\la^{-1}\left[  \|w\|_z^2 + (\al-1)\|w\|_z	\right] \\
	& \leq C\la^{-1}(\log \la)(\la^{-4} + (\al-1)^2). 
	\eeqas
Therefore from \eqref{eq:main} and \eqref{eq:w0} we have, 
\begin{equation*}
\la^{3-2\al}\ed E_\al(z)[T^\lambda] \leq C(\log \la)(\la^{-3}+ \la^{-1}(\al-1) + (\al-1)^2).
\end{equation*} 
By Proposition \ref{prop:dela} we have 
\begin{eqnarray*}
	\la^{3-2\al}\pl{}{\lambda}E_\al(z) = 8\pi \left(\mathcal{J}(a)\la^{-2} + (\al-1)\right)+(\al-1)O(\la^{-1})+ O(\la^{-3}) + (\al-1)^2O(1),
\end{eqnarray*}
which means 
\beqas
\mathcal{J}(a)\la^{-2} + (\al-1) = O(\la^{-3}\log\la) + (\al-1)\left[(O((\al-1)\log\la) + O(\la^{-1}\log \la )\right]+ O((\al-1)^2)).
\eeqas
Now Corollary \ref{cor:} gives $1\leq \la^{2\al-2}\leq 1+\eps$ so in particular $(\al-1)\log\la \leq C\eps$ and we can assume first 
$$O((\al-1)\log\la) + O(\la^{-1}\log \la ) \leq \fr12$$which readily gives the existence of some $C<\infty$ so that 
$$C^{-1}\la^{-2}\leq (\al-1)\leq C\la^{-2}$$
and then 
\beqas
\mathcal{J}(a)\la^{-2} + (\al-1) = O(\la^{-3}\log\la) =O((\al-1)^{\fr32}|\log(\al-1)|)
\eeqas
which is the first desired estimate in the degree one case. The final three estimates follow directly from this combined with \eqref{eq:w0}, Lemma \ref{lem:diff} and Remark \ref{rmk:2.28}.  
 
\end{proof}
\begin{proof}[Proof of Theorem \ref{thm:crit}]
Since we know that $|\mathcal{J}(a)|\la^{-2} \simeq (\al-1)$ applying \eqref{eq:main} to $T=T^A$ and using Lemmata \ref{lem:V} and \ref{lem:e4} combined with \eqref{eq:w0} we have first 
$$|\ed^2 E_\al(z)[T^A,W]|\leq C(\log\la)\la^{-3}$$
and second 
\beqas
	|\ed^2E_\al(z)[T^A,W]-\ed^2 E_\al(u_t)[T^A_t,W_t]& + \ed E_\al (u_t)[\ed P(u_t)\de_t T^A_t]|\\
	& \leq C(\log\la)\la^{-3} 
	\eeqas
	meaning that 
\begin{eqnarray*}
	\la^{4-2\al}\D_A (E_\al(z)) = \la^{4-2\al}\ed E_\al(z)[\D_A z] \leq  C\la^{-1}(\log \la). 
\end{eqnarray*}
Now by Proposition \ref{prop:dela} we have 
\begin{eqnarray*}
	\la^{4-2\al}\D_A (E_\al(z)) = 4\pi \D_A \mathcal{J}(a) + O(\la^{-1}),
\end{eqnarray*}
which means that 
$$\D_A \mathcal{J}(a)=O((\al-1)^{\fr12}|\log (\al-1)|)$$
as required. 
\end{proof}


\section{Energy expansions when restricted to $\ZZ$}\label{sec:exp}

Recalling the notation introduced in Remark \ref{rmk:not}: 

\begin{proposition}\label{prop:dela}
Let $\al\leq 2$, $\dl \leq 1$ and $z\in \mathcal{Z}^{\al,\dl}$. For $\la=\la(z)$ and $a=a(z)$ we have  
\begin{eqnarray*}
\la^{3-2\al}\pl{}{\lambda}E_\al(z) =	8\pi\left(\mathcal{J}(a)\la^{-2} + (\al-1)\right)+O(\la^{-1}(\al-1))+ O(\la^{-3}) + O((\al-1)^2).
\end{eqnarray*}

\end{proposition}

\begin{proof}
		
	Using \eqref{eq:rough_la_away} and integrating \eqref{eq:1st}
\beqas
\pl{}{\lambda}E_\al(z) &= -\al\int_{\Sigma} (2+|\D z|^2)^{\al-1}\Dl z \cdot \pl{z}{\la} \id \Sigma  \\
&\qquad - \al(\al-1)\int_\Sigma 	(2+|\D z|^2)^{\al-2}\D |\D z|^2\cdot \D z \cdot \pl{z}{\la}\id \Sigma \\
&= -\al\int_{\DD_r} (2+|\D^g z|_g^2)^{\al-1}\Dl_g z \cdot \pl{z}{\la} \id x_g\\
&\qquad - \al(\al-1)\int_{\DD_r} 	(2+|\D^g z|_g^2)^{\al-2}\D^g |\D^g z|_g^2\cdot_g \D^g z \cdot \pl{z}{\la} \id x_g   + O(\la^{-4}). \\
&= \underbrace{-\al\int_{\DD_r} (2+|\D^g z|_g^2)^{\al-1}\Dl z \cdot \pl{z}{\la} \id x}_{{\rm I}}\\
&\qquad \underbrace{- \al(\al-1)\int_{\DD_r} 	(2+|\D^g z|_g^2)^{\al-2}\D |\D^g z|_g^2\cdot \D z \cdot \pl{z}{\la} \id x}_{{\rm II}}    + O(\la^{-4})
\eeqas
where we have used that $\Dl_g f \id x_g = \Dl f \id x$ and furthermore $\D^g f_1 \cdot_g \D^g f_2 \id x_g = \D f_1 \cdot \D f_2 \id x$. 

\paragraph{Dealing with ${\rm I}$:}
For the first term above we have from \eqref{eq:zr},  
$$\Dl z = \Dl \pi_\la +\Dl j^\top + \Dl K\quad\text{and}\quad \pl{z}{\la} = \de_\la \pi_\la +\de_\la j^\top + \de_\la K.$$ 

By \eqref{eq:pop}  
\beqna\label{main_term1}
	\Dl z \cdot \de_\la z &=& \Dl \pi_\la \cdot \de_\la j^\top +\Dl j^\top \cdot \de_\la \pi_\la +\Dl j^\top \cdot  \de_\la j^\top+ (\Dl z -\Dl K)\cdot \de_\la K   + \Dl K \cdot \de_\la z\nn   \\
	&=& \Dl \pi_\la \cdot \de_\la j^\top +\Dl j^\top \cdot \de_\la \pi_\la + O\left(\la^{-3}\frac{\la|x|}{1+\la^2|x|^2}\right),
\eeqna  
using \eqref{eq:rough1a}, \eqref{eq:rough2}, \eqref{eq:rough3} and \eqref{eq:rough}. 
 
The first term above may be written 
\begin{eqnarray*}
	\Dl \pi_\la \cdot \de_\la j^\top  &=& -|\D \pi_\la|^2\pi_\la\cdot \left[\left(\de_\la j\right)^\top - (j \cdot \pi_\la) \de_\la \pi_\la - (j \cdot \de_\la \pi_\la)  \pi_\la  \right] \\
	&=& |\D \pi_\la|^2  (j \cdot \de_\la \pi_\la)
\end{eqnarray*}
and the second term, since $j$ is harmonic (see Remark \ref{rmk:jharm}) 
\begin{eqnarray*}
\Dl j^\top \cdot \de_\la \pi_\la &=& -\Dl((j \cdot \pi_\la) \pi_\la) \cdot  \de_\la \pi_\la	 \\
&=& -2 \D(j \cdot \pi_\la) \cdot \D \pi_\la \cdot \de_\la \pi_\la \\
&=& - |\D \pi_\la|^2(j \cdot  \de_\lambda \pi_\la)  - \la^{-1}[(x^1\de_1 j + x^2\de_2 j)\cdot \pi_\la  ]|\D \pi_\la|^2
\end{eqnarray*}
using that $\pi_\la$ is conformal (see \eqref{eq:pop}) and $\de_\la \pi_\la = \la^{-1} x\cdot \D \pi_\la (x)$. 

Thus \eqref{main_term1} becomes 
\begin{equation*}
	\Dl z \cdot \de_\la z = - \la^{-1}[x^i\de_i j \cdot \pi_\la  ]|\D \pi_\la|^2  + O\left(\la^{-3}\frac{\la|x|}{1+\la^2|x|^2}\right) .
\end{equation*}   
In particular using \eqref{eq:dz2}, \eqref{eq:pop} and Corollary \ref{cor:}: 
\beq\label{eq:lab}
(2+|\D^g z|_g^2)^{\al-1}\leq C(1+|\D\pi_\la|^{2\al-2})\leq C(1+\la^{2\al-2}) \leq C
\eeq 
giving 
\beqa\label{eq:left}
	{\rm I} &= -\al\int_{\DD_r} (2+|\D^g z|_g^2)^{\al-1}\Dl z \cdot \pl{z}{\la} \id x \\
	 &=\al \int_{\DD_r} (2+|\D^g z|_g^2)^{\al-1} [\la^{-1}x^i\de_i j\cdot \pi_\la  ]|\D \pi_\la|^2  \id x + O(\la^{-4})
\eeqa
as can be checked directly.

Recall again from \eqref{eq:dz2} that $|\D^g z|_g^2 = |\D \pi_\la|^2 (c_\gamma^{-1}+ \chi) $ where $\chi=O(\la^{-2}  + |x|^2)$ which gives (by a few applications of the mean value formula \eqref{eq:mv}), when $|x|\leq \la^{-\fr12}$
 \beqas
 	(2+|\D^g z|_g^2)^{\al-1} &= |\D\pi_\la|^{2\al-2}(c_\gamma^{-1}+\chi)^{\al-1}(1+ 2(c_\gamma^{-1}+\chi)^{-1}|\D\pi_\la|^{-2} )^{\al-1}  \\
 	&= c_\gamma^{1-\al}|\D \pi_\la|^{2\al-2}+O((\al-1)|\D \pi_\la|^{2\al-4} ) 
 \eeqas
 since $\chi\leq |\D\pi_\la|^{-2}$. 
 
 Thus
 \begin{equation}\label{eq:dz2al-1}
 	(2+|\D^g z|_g^2)^{\al-1} = \twopartdef{c_\gamma^{1-\al}|\D \pi_\la|^{2\al-2} +e_1}{|x|\leq \la^{-\fr12}}{e_1}{|x|>\la^{-\fr12}}
 \end{equation}
 where \beq\label{eq:e1}
 e_1= \twopartdef{O((\al-1)|\D \pi_\la|^{2\al-4} ) }{|x|\leq \la^{-\fr12}}{1+O(\al-1)}{|x|>\la^{-\fr12}.}\eeq

 Now \eqref{eq:left} and \eqref{eq:dz2al-1} give 
 \beqa\label{eq:laI}
  {\rm I}
 &=\al \int_{\DD_{\la^{-\fr12}}} c_\gamma^{1-\al}|\D \pi_\la|^{2\al}[\la^{-1}x^i\de_i j\cdot \pi_\la  ]\id x  +\int_{\DD_r} e_1[\la^{-1}x^i\de_i j\cdot \pi_\la  ]|\D \pi_\la|^{2}\id x + O(\la^{-4}).
 \eeqa
 Recall that $j (x)=j_{\la,a}(x) =\left(\tfrac{2}{\lambda}\D_y J_a(x,0)-\tfrac{2}{\lambda}\D_y J_a(0,0), 0\right)$, and $\pi_\la^{1,2}=\fr{2\la x}{1+\la^2|x|^2}$ so, first by Taylor expansion about $x=0$ we have
$$(x^1\de_1 j + x^2\de_2 j) = x^1\de_1 j(0) + x^2\de_2 j(0) +  \la^{-1}O(|x|^2)$$
and then
\beqa\label{eq:laI1}
\la^{-1}(x^1\de_1 j + x^2\de_2 j)\cdot \pi_\la  &= \fr{4\la^{-1}}{1+\la^2|x|^2} \left((x^1)^2 \de_{x^1}\de_{y^1} J(0,0) + (x^2)^2 \de_{x^2}\de_{y^2} J(0,0) \right) \\
& \quad  +\fr{4\la^{-1} x^1x^2f_0}{1+\la^2|x|^2}  + O\left(\frac{\la^{-1}|x|^3}{1+\la^2|x|^2}\right)	
\eeqa
where $f_0$ is the constant $f_0=\de_{x^2}\de_{y^1} J(0,0) +  \de_{x^1}\de_{y^2} J(0,0)$.

In particular 
\beq\label{eq:laI2}\la^{-1}x^i\de_i j\cdot \pi_\la  =O(\la^{-3}) \overset{\eqref{eq:e1}}{\implies}\int_{\DD_r}e_1[\la^{-1}x^i\de_i j\cdot \pi_\la  ]|\D \pi_\la|^{2}\id x = O(\la^{-4}(\al-1))+O(\la^{-6})\eeq
 and 
\beq\label{eq:laI3}\int_{\DD_{\la^{-\fr12}}} |\D \pi_\la|^{2\al}\frac{\la^{-1}|x|^3}{1+\la^2|x|^2}  \id x= O(\la^{-4})\eeq
 as can be checked directly.
 
Using \eqref{eq:laI1}-\eqref{eq:laI3}, \eqref{eq:laI} becomes 
\beqas
 {\rm I}
 &=\al \int_{\DD_{\la^{-\fr12}}} c_\gamma^{1-\al}|\D \pi_\la|^{2\al}\fr{4\la^{-1}}{1+\la^2|x|^2} \left((x^1)^2 \de_{x^1}\de_{y^1} J(0,0) + (x^2)^2 \de_{x^2}\de_{y^2} J(0,0) \right) \id x  \\
 & \quad + \al \int_{\DD_{\la^{-\fr12}}} c_\gamma^{1-\al}|\D \pi_\la|^{2\al}\fr{4\la^{-1}x^1x^2f_0}{1+\la^2|x|^2}  \id x + O(\la^{-4}).  
 \eeqas

Using the facts that, for radial $f$ and a disc of any radius centred at the origin   
$$\int_{\DD} x^1 x^2 f(|x|)\id x = 0, \quad \text{and} \quad \int_{\DD} (x^1)^2 f(|x|)\id x  = \int_{\DD} (x^2)^2 f(|x|)\id x  = \fr12\int_{\DD} |x|^2 f(|x|)\id x $$ we see that, setting $\mathcal{J}(a): =(\de_{x^1}\de_{y^1} J_a(0,0) + 
 \de_{x^2}\de_{y^2} J_a(0,0)),$ which coincides with Definition \ref{defSig} (see Appendix \ref{App:J})
\beqas
  {\rm I} &= 	2\al c_\gamma^{1-\al} 8^{\al}\la^{2\al -3} \mathcal{J}(a)  \int_{\DD_{\la^{-\fr12}}} \frac{(\la|x|)^2}{(1+\la^2|x|^2)^{2\al +1}}\id x  + O(\la^{-4})\\
  &= 2\al c_\gamma^{1-\al} 8^{\al}\la^{2\al -5} \mathcal{J}(a)\int_{\DD_{\la ^{\fr12}}}\frac{|x|^2}{(1+|x|^2)^{2\al +1}}\id x  + O(\la^{-4})\\
 &= \frac{8^{\al}\pi c_\gamma^{1-\al}}{2\al-1}\mathcal{J}(a)\la^{2\al-5}  + O(\la^{-4}) \\
 &= 8\pi \mathcal{J}(a)\la^{2\al-5} +O(\la^{-4}) + O(\la^{-3}(\al-1)).
\eeqas

\paragraph{Dealing with ${\rm II}$:}
From \eqref{eq:dz} we have, on $\DD_r$
\beq\label{eq:dzII}
\D z = \D\pi_\la  + O(\la^{-1})
\eeq 
and from \eqref{eq:zr}, \eqref{eq:rough2}
\beqa
\de_\la z &= \de_\la \pi_\la  + \de_\la j^\top  + \de_\la K \\
&= \de_\la \pi_\la  + O(\la^{-2}|x|).
\eeqa
From \eqref{eq:zr} and the definition of $|\D^g z|_g^2$ it follows that
\beqas
\D |\D^g z|_g^2 &= c_\gamma^{-1} \D |\D z|^2 + O(|x|^2)\D |\D z|^2 + O(|x|)|\D z|^2\\
&=  c_\gamma^{-1}\D|\D \pi_\la|^2 + O(|x|^2)\D|\D \pi_\la|^2 + O(1)\left\{\D^2j^\top \cdot \D j^\top + \D^2 K \cdot \D K \right. \\
&\quad  +  \D^2 \pi_\la \cdot \D j^\top +  \D \pi_\la \cdot \D^2 j^\top   + \D^2 \pi_\la \cdot \D K  + \D \pi_\la \cdot \D^2 K + \D^2 j^\top \cdot \D K \\  
&\quad \left. +  \D j \cdot \D^2 K\right\}+ O(|x||\D z|^2). 
\eeqas After applying \eqref{eq:pop}-\eqref{eq:rough1a}, 
\beq\label{eq:ddz2}
\D |\D^g z|_g^2 =c_\gamma^{-1} \D |\D \pi_\la|^2  + O\left(\frac{\la}{(1+\la^2|x|^2)^{\frac32}}\right).
\eeq

Using \eqref{eq:dzII}-\eqref{eq:ddz2}, along with \eqref{eq:pop} and $\de_\la \pi_\la = \la^{-1}x\cdot \D \pi_\la$ we recover  
\beqna\label{eq:IIla}
{\rm II} &= &- \al (\al-1) \int_{\DD_r} (2+|\D^g z|_g^2)^{\al-2} \D |\D^g z|_g^2\cdot \D z \cdot \de_\la z \id x   \nn \\
&=& - \al (\al-1) c_\gamma^{-1} \int_{\DD_r} (2+|\D^g z|_g^2)^{\al-2}\left[\D |\D \pi_\la|^2 \cdot \D \pi_\la \cdot \de_\la \pi_\la  + O\left(\frac{\la^2|x|}{(1+\la^2|x|^2)^{\frac72}}\right) \right] \id x. \qquad  	
\eeqna
  
Similarly to \eqref{eq:dz2al-1} and since $\chi|\D \pi_\la|^{2\al-4} \leq C |\D \pi_\la|^{2\al-6}$ we have 
  \begin{equation}\label{eq:dz2al-2}
 	(2+|\D^g z|_g^2)^{\al-2} = \twopartdef{c_\gamma^{2-\al}|\D \pi_\la|^{2\al-4} +e_2}{|x|\leq \la^{-\fr12}}{e_2}{|x|>\la^{-\fr12}}
 \end{equation}
 where $$e_2= \twopartdef{O(|\D \pi_\la|^{2\al-6})}{|x|\leq \la^{-\fr12}}{O(1)}{|x|>\la^{-\fr12.}}$$

Thus \eqref{eq:dz2al-2} gives 
 \begin{eqnarray*}
 \al(\al-1)\int_{\DD_r} (2+|\D^g z|_g^2)^{\al-2} \frac{\la^2|x|}{(1+\la^2|x|^2)^{\frac72}}	 \id x= (\al-1)O(\la^{-3}\log \la)
 \end{eqnarray*}
 and therefore \eqref{eq:IIla} becomes 
 
\beq\label{eq:IIla1}
{\rm II} =- \al (\al-1)c_\gamma^{-1} \int_{\DD_r} (2+|\D^g z|_g^2)^{\al-2}\D |\D \pi_\la|^2 \cdot \D \pi_\la \cdot \de_\la \pi_\la \id x + (\al-1)O(\la^{-3}\log \la).
\eeq
By direct computation  
\begin{equation*}
\D |\D \pi_\la|^2 \cdot \D \pi_\la \cdot \de_\la \pi_\la = -128 \frac{\la^5 |x|^2}{(1+\la^2|x|^2)^5}, 	
\end{equation*}
so by \eqref{eq:dz2al-2} again it can be checked that: 

\beqas
{\rm II} &= - \al (\al-1)c_\gamma^{1-\al}\int_{\DD_{\la^{-\fr12}}}\D |\D \pi_\la|^2 \cdot \D \pi_\la \cdot \de_\la \pi_\la |\D \pi_\la|^{2\al-4} \id x + (\al-1)O(\la^{-2}), 
\eeqas 
and \eqref{eq:IIla1} becomes 
\beqas
{\rm II} &= 2^{3\al+1}\al(\al-1)\la^{2\al-1}c_\gamma^{1-\al}\int_{\DD_{\la^{-\fr12}}} \frac{\la^2|x|^2}{(1+\la^2|x|^2)^{2\al+1}}\id x+ (\al-1)O(\la^{-2})\\
&= \fr{8^{\al}\pi c_\gamma^{1-\al}}{2\al-1}(\al-1)\la^{2\al-3} + (\al-1)O(\la^{-2}) \\
&= 8\pi (\al-1) \la^{2\al-3} + (\al-1)O(\la^{-2}) + (\al-1)^2O(\la^{-1}). 
\eeqas
\end{proof}

\begin{proposition}\label{prop:dea}
Let $\al\leq 2$, $\dl \leq 1$ and $z\in \mathcal{Z}^{\al,\dl}$. Then for a surface of genus $\gamma \geq 2$, for $\la=\la(z)$ and $a=a(z)$ we have
\begin{eqnarray*}
\la^{4-2\al}\D_A (E_\al(z)) = 4\pi \D_A \mathcal{J}(a) +   O(\la^{-1})+O(\la(\al-1)). 	
\end{eqnarray*}

\end{proposition}
\begin{proof}
	From \eqref{eq:rough_la_away} and \eqref{eq:roughAaway}, similarly to the beginning of the proof of Proposition \ref{prop:dela} we have
\beqas
\D_A (E_\al(z)) &= -\al\int_{\Sigma} (2+|\D z|^2)^{\al-1}\Dl z \cdot \D_A z \id \Sigma  \\
&\quad  - \al(\al-1)\int_\Sigma 	(2+|\D z|^2)^{\al-2}\D |\D z|^2\cdot \D z \cdot \D_A z \id \Sigma  \\
&= \underbrace{-\al\int_{\DD_r} (2+|\D^g z|_g^2)^{\al-1}\Dl z \cdot \D_A z \id x}_{{\rm III}}\\
&\qquad \underbrace{- \al(\al-1)\int_{\DD_r} 	(2+|\D^g z|_g^2)^{\al-2}\D |\D^g z|_g^2\cdot \D z \cdot \D_A z \id x}_{{\rm IV}}   + O(\la^{-3}). 
\eeqas
\paragraph{Dealing with ${\rm III}$:}
For the first term above we have 
$$\Dl z = \Dl \pi_\la +\Dl j^\top  + \Dl K$$ 
and trivially 
$$\D_A z = \D_A \pi_\la + \de_{A}j_{a}^{\top} + \de_{A} K$$
giving
\begin{eqnarray}\label{eq:dea}
	\Dl z \cdot \D_A z &=& \Dl \pi_\la \cdot \D_A j^\top +\Dl j^\top \cdot \D_A\pi_\la +\Dl j^\top \cdot  \D_A j^\top+ (\Dl z -\Dl K)\cdot \D_A K   + \Dl K \cdot \D_A z \nn \\
	&=& \Dl \pi_\la \cdot \D_A j^\top -\Dl j^\bot \cdot \D_A\pi_\la + O\left(\la^{-1}\frac{\la|x|}{(1+\la^2|x|^2)^2}\right) 
\end{eqnarray}
since $j$ is harmonic (Remark \ref{rmk:jharm})and from 
\eqref{eq:rough1a}, \eqref{eq:rough}, \eqref{eq:roughA}, \eqref{eq:roughA1}. 

The first term above may be written
\begin{eqnarray*}
	\Dl \pi_\la \cdot \D_A j^\top  &=& -\pi_\la |\D \pi_\la|^2 \cdot \left[\left(\D_A j_{a}\right)^\top - (j \cdot \pi_\la) \D_A \pi_\la - (j \cdot \D_A \pi_\la)  \pi_\la  \right] \\
	&=& |\D \pi_\la|^2  (j \cdot \D_A \pi_\la) \\
	&=& - |\D \pi_\la|^2  (j \cdot A^i\de_i \pi_\la) + O\left(\la^{-1}\frac{\la|x|}{(1+\la^2|x|^2)^2}\right)
\end{eqnarray*}
from \eqref{eq:deApila}, the definition of $\pi_\la$ and that $|j|\leq C|x|\la^{-1}$.  

Again using \eqref{eq:deApila} the second term becomes
\begin{eqnarray*}
\Dl j^\top \cdot \D_A \pi_\la &=& -\Dl((j \cdot \pi_\la) \pi_\la) \cdot  \D_A \pi_\la	 \\
&=& -2 \D(j \cdot \pi_\la) \cdot \D \pi_\la \cdot \D_A \pi_\la \\
&=& 2\D(j\cdot\pi_\la)\cdot\D \pi_\la \cdot A^i\de_i\pi_\la +O\left(\la^{-1}\frac{\la|x|}{(1+\la^2|x|^2)^2}\right)\\
&=& |\D \pi_\la|^2 (j\cdot A^i\de_i \pi_\la + A^i\de_i j \cdot \pi_\la) +O\left(\la^{-1}\frac{\la|x|}{(1+\la^2|x|^2)^2}\right)
\end{eqnarray*}
since $\de_1 \pi_\la \cdot \de_2 \pi_\la =0$, $|\de_1 \pi_\la|^2 =|\de_2 \pi_\la|^2 = \fr12 |\D \pi_\la|^2$.

Thus \eqref{eq:dea} may be written 
\begin{equation*}
	\Dl z \cdot \D_A z = |\D \pi_\la|^2A^i\de_i j \cdot \pi_\la + O\left(\la^{-1}\frac{\la|x|}{(1+\la^2|x|^2)^2}\right). \end{equation*}

Using again $(2+|\D^g z|_g^2)^{\al-1}\leq C$ (see \eqref{eq:lab}) we see that 
\begin{eqnarray*}
	-\al\int_{\DD_r} (2+|\D^g z|_g^2)^{\al-1}\la^{-1}\frac{\la|x|}{(1+\la^2|x|^2)^2}\id x  = O(\la^{-3})
\end{eqnarray*}
and therefore that 
\begin{eqnarray*}
{\rm III}&=&	-\al\int_{\DD_r} (2+|\D^g z|_g^2)^{\al-1}\Dl z_{\lambda,a} \cdot \D_A z \id x \nn\\ 
&=&-\al \int_{\DD_r} (2+|\D^g z|_g^2)^{\al-1} |\D \pi_\la|^2A^i\de_i j \cdot \pi_\la \id x + O(\la^{-3}). 
\end{eqnarray*}
Recall that $j(x)=j_{\lambda,a} (x) =\left(\tfrac{2}{\lambda}\D_y J_a(x,0)-\tfrac{2}{\lambda}\D_y J_a(0,0), 0\right)$, and $\pi_\la^{1,2}=\fr{2\la x}{1+\la^2|x|^2}$ so. Combining this with \eqref{eq:dz2al-1}-\eqref{eq:e1}, after a short calculation we obtain   
 \beqa\label{eq:III}
{\rm III}
 &=-4 ^{1-\al}\al \int_{\DD_{\la^{-\fr12}}} |\D \pi_\la|^{2\al}A^i\de_i j \cdot \pi_\la \id x -\al \int_{\DD_r \sm \DD_{\la^{-\fr12}}}|\D \pi_\la|^{2}A^i\de_i j \cdot \pi_\la \id x  \\ &\quad  +O((\al-1)\la^{-2}). \eeqa

By Taylor expansion about $x=0$ we have
$$A^i\de_i j = A^i\de_i j(0) + x^1A^1\de^2_{11}j(0)+(x^1A^2+x^2A^1)\de^2_{12}j(0) + x^2A^2\de^2_{22}j (0) +  O(\la^{-1}|x|^2)$$
which gives
\begin{eqnarray*}
A^i\de_i j \cdot \pi_\la &=& \fr{4}{1+\la^2|x|^2} (x^1)^2 (A^1\de_{x^1} \de_{x^1}\de_{y^1} J(0,0) +A^2\de_{x^2}\de_{x^1}\de_{y^1} J(0,0))\\
&& +\fr{4}{1+\la^2|x|^2}(x^2)^2 (A^1\de_{x^1}\de_{x^2}\de_{y^2} J(0,0) +A^2  \de_{x^2}\de_{x^2}\de_{y^2} J(0,0))  \\
&& + f_0(|x|)x^1x^2 + f_i(|x|)x^i + O\left(\frac{|x|^3}{1+\la^2|x|^2}\right)
\end{eqnarray*}
 for some suitable radial functions $f_0, f_1, f_2$. 

Using also the facts: $\int_{\DD_r} x^i f(|x|) = 0$ for each $i$, $\int_{\DD_r} x^1 x^2 f(|x|) = 0$, $\int_{\DD_r} (x^1)^2 f(|x|) = \int_{\DD_r} (x^2)^2 f(|x|)$ and from Lemma \ref{lem:deJ} 
$$-\fr12\D_A \mathcal{J}(a) = A^i\de_{x^i} [\de_{x^1}\de_{y^1} J_a(0,0) + 
 \de_{x^2}\de_{y^2} J_a(0,0)]  $$
we are left with (after a series of direct calculations from \eqref{eq:III})  
\begin{eqnarray*}
  {\rm III} = 	4\pi \D_A \mathcal{J}(a)\la^{2\al -4} + O((\al-1)\la^{-2})+  O(\la^{-3}).
\end{eqnarray*}

\paragraph{Dealing with ${\rm IV}$:}
Combining \eqref{eq:dzII} and \eqref{eq:ddz2} with 
\beqas
\D_A z &= \D_A \pi_\la  + \D_A j^\top  + \D_A K \\
&= \D_A \pi_\la  + O(\la^{-1}) \\
&= -A^i\de_i\pi_\la + O(\la^{-1})
\eeqas
we obtain 
\beqas 
\D |\D^g z|_g^2\cdot \D z \cdot \D_A z 
&= -\fr14\D |\D \pi_\la|^2 \cdot \D\pi_\la \cdot (A^i\de_i\pi_\la) + O\left(\frac{\la^3}{(1+\la^2|x|^2)^{\fr72}}\right) \\
&= \fr18  |\D \pi_\la|^2 A^i\de_i |\D \pi_\la|^2 + O\left(\frac{\la^3}{(1+\la^2|x|^2)^{\fr72}}\right)
\eeqas 
since $\pi_\la$ is conformal. 

Using \eqref{eq:dz2al-2} we have that 
\beqs
 - \al(\al-1)\int_{\DD_r} 	(2+|\D^g z|_g^2)^{\al-2} \frac{\la^3}{(1+\la^2|x|^2)^{\fr72}} \id x= (\al-1) O(\la^{-1}),
\eeqs
giving 
\beq\label{eq:IV} 
{\rm IV} = -\fr18 \al(\al-1)\int_{\DD_r} 	(2+|\D^g z|_g^2)^{\al-2}|\D \pi_\la|^2 A^i\de_i |\D \pi_\la|^2 \id x + O(\la^{-1}(\al-1)).
\eeq 
Now, since $$A^i\de_i |\D \pi_\la|^2 = f(|x|)(A^1x^1 +A^2 x^2)$$ for some radial function $f$ we note that 
$$
\int_{\DD} F(|x|)|\D \pi_\la|^2 A^i\de_i |\D \pi_\la|^2 \id x = 0$$
again for any radial function $F$. Thus it remains to suitably estimate only the non-radial terms of $(2+|\D^g z|_g^2)^{\al-2}$ in \eqref{eq:IV}. Instead of \eqref{eq:dz2al-2} we again use $|\D^g z|_g^2 = |\D \pi_\la|^2(\tfrac14+\chi)$ and the mean value theorem to give 
\beqas
(2+|\D^g z|_g^2)^{\al-2} - (2+\tfrac14|\D\pi_\la|^2)^{\al-2} &= \tfrac14|\D \pi_\la|^{2\al-4}\left( (2|\D \pi_\la|^{-2} + 1 + \chi)^{\al-2} - (2|\D \pi_\la|^{-2} + 1)^{\al-2}\right) \\
&= \chi |\D \pi_\la|^{2\al-4}.
\eeqas

Since $(2+|\D\pi_\la|^2)^{\al-2}$ is radial we are left with 
\beqas
-\fr18\al(\al-1)\int_{\DD_r}(2+|\D^g z|_g^2)^{\al-2}|\D \pi_\la|^2 A^i\de_i |\D \pi_\la|^2 \id x&= (\al-1)\int_{\DD_r}\chi|\D \pi_\la|^{2\al-2} A^i\de_i |\D \pi_\la|^2 \id x\\
&= O(\la^{-1}(\al-1))
\eeqas 
as can be checked directly using e.g. \eqref{eq:pop}. 
Thus we have ${\rm IV}= O(\la^{-1}(\al-1))$ and we are done. 
\end{proof}


\section{Estimates for $\ed E_\al$ and $\ed^2 E_\al$ close to $\ZZ$}\label{sec:est}
Here we prove that the desired `error' terms appearing in the comparisons of $0=\ed E_\al(u)[\cdot]$ and $\ed E_\al(z)[\cdot]$ are lower order in relation to the expansions of Section \ref{sec:exp}. We first prove the required estimates on $\ZZ$ itself:  
\subsection{Estimates on $\ZZ$} 
\begin{lemma}\label{lem:V}
	Let $\al\leq 2$, $\dl \leq 1$ and $z\in \mathcal{Z}^{\al,\dl}$, then there exists $C=C(\Sigma)<\infty$ so that for any $V\in \Gamma^2(z)$ we have 
	$$|\ed E_\al(z)[V]|\leq C(\log \la)^\fr12(\la^{-2} +(\al-1))\|V\|_z.$$
Furthermore 
	$$|\ed^2E_\al(z)[\de_\la z, V]| \leq C(\log\la)^\fr12(\la^{-2} + \la^{-1}(\al-1))\|V\|_z$$
	and 
	
	$$|\ed^2E_\al(z)[\D_A z,V]| \leq C (\log\la)^{\fr12}(\la^{-1} + \lambda (\al-1))\|V\|_z.$$
\end{lemma}
\begin{proof}
From \eqref{eq:lab} we know $(2+|\D z|^2)^{\al-1}\leq C$. 
Thus 
\beqna\label{eq:51}
|\ed E_\al (z) [V]| &=& \left|-\al \int_{\Sigma} (2+|\D z|^2)^{\al-1} \Dl z \cdot V \id \Sigma  \right.\nn\\
&&\quad \left. - \al(\al-1) \int_\Sigma (2+|\D z|^2)^{\al-2}\D |\D z|^2 \cdot \D z \cdot V \id \Sigma \right| \nn\\
&\leq & C\int_{\Sigma} |\Dl z\cdot V|+ (\al-1)|V|\frac{|\D |\D z|^2\cdot \D z |}{2+|\D z|^2} \id \Sigma  \nn \\
&\leq & C\int_{\DD_r} |\Dl z\cdot V| + (\al-1)|V|\frac{|\D |\D ^g z|_g^2 \cdot \D z |}{2+|\D^g z|_g^2} \id x +C\la^{-2} \int_{\Sigma\sm U_r(a)} |V| \id \Sigma \nn \\
&\leq &   C\int_{\DD_r} |\Dl (\pi_\la + j^\top)\cdot V| + (\al-1)|V|\frac{|\D |\D ^g z|_g^2 \cdot \D z |}{2+|\D^g z|_g^2} \id x +C\la^{-2} \int_{\Sigma} |V| \id \Sigma\nn \\
&\leq &C\int_{\DD_r} |\Dl (\pi_\la + j^\top)\cdot V| + (\al-1)|V|\frac{|\D |\D ^g z|_g^2 \cdot \D z |}{2+|\D^g z|_g^2} \id x +C(\log\la)^\fr12\la^{-2} \|V\|_z \qquad  
\eeqna 
using \eqref{eq:rouhgla} and \eqref{eq:rough_la_away} and Remark \ref{rmk:2.28}.

Using \eqref{eq:rough} we note that 
\beqas
\frac{|\D |\D ^g z|_g^2 \cdot \D z |}{2+|\D^g z|_g^2}&\leq  C\frac{|\D^2 z||\D z|^2 + |x||\D z|^3}{2+|\D z|^2} \leq C(|\D^2 z|+ |x||\D z|)\\
&\leq C\rho_z \frac{\la}{(1+\la^2|x|^2)^{\fr12}} 
\eeqas
which yields 
\beqna\label{eq:52}
  \int_{\DD_r} (\al-1)|V|\frac{|\D |\D ^g z|_g^2 \cdot \D z |}{2+|\D^g z|_g^2} \id x  
&\leq& C(\al-1)\left(\int_{\DD_r} \rho_z^2 |V|^2\right)^\fr12\left( \int_{\DD_r} \fr{\la^2}{(1+\la^2|x|^2)}\right)^\fr12 \nn\\
&\leq& C(\log \la)^{\fr12}(\al-1)\|V\|_z.   
\eeqna

Finally since $V\in \Gamma^2(z)$ we need only estimate $(\Dl \pi + \Dl j^\top)^{\top_z}$ which we do now. From \eqref{eq:pop}, \eqref{eq:Lapjtop} and \eqref{eq:jpi} we have  
\beqas
 \Dl \pi + \Dl j^\top &=  O\left(\frac{\la^2}{(1+\la^2|x|^2)^2}\right)\pi_\la + O\left(\frac{|x|}{1+\la^2|x|^2}\right)\D \pi_\la   
\eeqas
and also, since $\pi_\la \cdot j^\top = 0$
$$\pi^{\top_z}=\pi_\la - (\pi_\la\cdot z) z = -j^\top - K -(\pi_\la\cdot K)(\pi_\la + j^\top + K)=O(\la^{-1}|x|)$$
which gives 
\beqs
(\Dl \pi + \Dl j^\top)^{\top_z} =O\left( \frac{|x|}{1+\la^2|x|^2}\right)\rho_z.    
\eeqs
We thus have 
\beqna\label{eq:53}
C\int_{\DD_r} |\Dl (\pi_\la + j^\top)\cdot V| &\leq& C\left(\int_{\DD_r} \rho_z^2 |V|^2\right)^\fr12\left(\int_{\DD_r} \fr{|x|^2}{(1+\la^2|x|^2)^2}\right)^\fr12 \nn\\
&\leq & C\la^{-2}(\log \la)^{\fr12} \|V\|_z.
\eeqna
Putting together \eqref{eq:51}, \eqref{eq:52} and \eqref{eq:53} gives the first estimate. 

For the second one we use \eqref{eq:rough_la_away} to give 
\beqnas
|\ed^2 E_\al(z)[\de_\la z,V]| &=&\left| \al\int_\Sigma (2+|\D z|^2)^{\al-1}( \D \de_\la z \cdot \D V - |\D z|^2 \de_\la z\cdot V)\right. \\
&& \left.+\al(\al-1)\int_\Sigma (2+|\D z|^2)^{\al-2}(\D z \cdot \D \de_\la z)(\D z \cdot \D V) \right|\\
&\leq & \al\left|\int_{\DD_r}(2+|\D^g z|_g^2)^{\al-1}( \D \de_\la z \cdot \D V - |\D z|^2 \de_\la z\cdot V)\id x \right|\\
&& + C(\al-1)\int_{\DD_r} |\D \de_\la z| |\D V| + C\la^{-2}\int_{\Sigma\sm U(r)} |\D V| + \rho_z |V|\eeqnas	
where the last term is bounded above by $\la^{-2}|\Sigma| \|V\|_z$ by Cauchy-Schwartz.

On $\DD_r$ we have $\de_\la z = \de_\la \pi_\la  + O(|x|\la^{-2})$, $|\D z|^2 = |\D \pi_\la|^2(1+\chi)$ and   
$$\D \de_\la z = \D \de_\la \pi + \D \de_\la j^\top + \D \de_\la K = \D\de_\la \pi_\la + O(\la^{-2})$$
so the above becomes, using also \eqref{eq:rough} and  \eqref{eq:rouhgla}
\beqnas
|\ed^2 E_\al(z)[\de_\la z,V]| &\leq & \al\left| \int_{\DD_r}(2+|\D^g z|_g^2)^{\al-1}( \D \de_\la \pi_\la \cdot \D V - |\D\pi_\la|^2\de_\la  \pi_\la \cdot V)\id x\right|\nn \\ 
&& + C\int_{\DD_r}(\la^{-2} + |x|^2)|\D \pi_\la|^2  |\de_\la z| |V| +|x|\la^{-2}|\D \pi_\la|^2||V|  \nn \\
&&+C(\al-1) \int_{\DD_r} (1+\la^2|x|^2)^{-1}|\D V| + C\la^{-2} \|V\|_z  \nn \\
&\leq &  \al\left| \int_{\DD_r}(2+|\D^g z|_g^2)^{\al-1}( \D \de_\la \pi_\la \cdot \D V - |\D\pi_\la|^2\de_\la  \pi_\la \cdot V) \id x\right|\nn \\
&& +  C(\la^{-2} + (\al-1)\la^{-1})\|V\|_z \nn \\ 
&\leq& \left|\al \int_{\DD_r}(2+|\D^g z|_g^2)^{\al-1}(- \Dl \de_\la \pi_\la \cdot  V - |\D\pi_\la|\de_\la  \pi_\la \cdot V)\right| \nn\\  
&& +\al(\al-1)\int_{\DD_r} (2+|\D^g z|_g^2)^{\al-2}|\D |\D^g z|_g^2 ||\D \de_\la \pi_\la ||V| \nn \\
&& + \al\int_{\de\DD_r} (2+|\D^g z|_g^2)^{\al-1} |\D \de_\la \pi_\la| |V|+ C(\la^{-2} + (\al-1)\la^{-1})\|V\|_z. \nn
\eeqnas

We estimate using \eqref{eq:pop} \eqref{eq:rough} and \eqref{eq:rouhgla}
$$(2+|\D^g z|_g^2)^{\al-2}|\D |\D^g z|_g^2 ||\D \de_\la \pi_\la |\leq C|\D \de_\la \pi_\la|\frac{|\D^2 z||\D z| + |x||\D z|^2}{2+|\D z|^2}\leq C(1+\la^2|x|^2)^{-\fr12}\rho_z$$
and also $\D \de_\la \pi_\la|_{\de \DD_r} = O(\la^{-2})$. Combining this with  Remark \ref{rmk:2.28}) the above reduces to 
\beqas
|\ed^2 E_\al(z)[\de_\la z,V]| &\leq  C\int_{\DD_r}|(\Dl \de_\la \pi_\la  + |\D\pi_\la|\de_\la  \pi_\la) \cdot V| \\
& \quad + C(\log\la)^{\fr12}(\la^{-2} + (\al-1)\la^{-1})\|V\|_z.
\eeqas
One checks, by direct calculation, that 
$$\Dl \de_\la \pi_\la  +|\D \pi_\la|^2 \de_\la \pi_\la  =-\fr2{\la} |\D \pi_\la|^2 \pi_\la^3 \pi_\la. $$
Thus the remaining term to estimate has an integrand of the form, using that $V\cdot z =0$ a.e.: 
$$\fr{2}{\la}|\pi_\la^3| |\D \pi_\la|^2 |\pi_\la \cdot V|  = \fr{2}{\la}|\pi_\la^3| |\D \pi_\la|^2| 
|(j+K)||V|\leq C\la^{-2}\rho_z|V|$$
which finishes the second estimate.  

For the third estimate we do the same calculations for the second but using instead \eqref{eq:daz}-\eqref{eq:roughAaway} to give 
 first 
 \beqnas
|\ed^2 E_\al(z)[\D_A z,V]| &\leq&  \left|\al\int_{\DD_r}(2+|\D^g z|_g^2)^{\al-1}( \D \D_A z \cdot \D V - |\D z|^2 \D_A z \cdot V) \id x\right|\\
&& + C(\al-1)\int_{\DD_r} | \D \D_A z||\D V| + C\la^{-1}\int_{\Sigma\sm U(r)} |\D V| + \rho_z |V|\eeqnas	
and then, on $\DD_r$ we have $\D_A z = -A^j\de_j \pi_\la +O(|x|^2)|\D \pi_\la|$, $|\D z|^2 = |\D \pi_\la|^2(1+\chi)$ and 
$$\D \D_A z = -A^j \D \de_j \pi_\la+ O(\la^{-1}) + O(1)(1+\la^2|x|^2)^{-1}$$ meaning that (since $|\D A^j\de_j \pi_\la| \sim \la^{-1}$ on $\de\DD_r$),  
\beqnas
|\ed^2 E_\al(z)[\D_A z,V]| &\leq&  \left|\al\int_{\DD_r}(2+|\D^g z|_g^2)^{\al-1}( \D (A^j\de_j \pi_\la) \cdot \D V - |\D \pi_\la|^2 (A^j\de_j\pi_\la)\cdot V)\id x\right|  \nn \\
&& + C\int_{\DD_r} \chi|\D \pi_\la|^3 |V|+ C(\al-1)\int_{\DD_r} |\D^2 \pi_\la||\D V| + C\la^{-1}\|V\|_z   \nn \\
&\leq & \left|\al\int_{\DD_r}(2+|\D^g z|_g^2)^{\al-1}( \D (A^j\de_j \pi_\la) \cdot \D V - |\D \pi_\la|^2 (A^j\de_j\pi_\la)\cdot V)\id x\right|  \nn \\
&& + C(\la^{-1} + \la(\log\la)^{\fr12}(\al-1) )\|V\|_z \nn \\
&\leq& \al \left| \int_{\DD_r}(2+|\D^g z|_g^2)^{\al-1}( -\Dl (A^j\de_j \pi_\la) \cdot  V - |\D \pi_\la|^2 (A^j\de_j\pi_\la)\cdot V) \right| \nn \\
&& + C(\al-1)\int_{\DD_r} (2+|\D^g z|_g^2)^{\al-2}|\D |\D^g z|_g^2| |\D (A^j\de_j\pi_\la)| | V|  \nn \\
&& +C\int_{\de\DD_r} |\D (A^j\de_j\pi_\la)| |V|\nn \\
&&  +C(\la^{-1} + \la (\log\la)^{\fr12}(\al-1))\|V\|_z \nn \\
&\leq &  C\int_{\DD_r}| (\Dl (A^j\de_j \pi_\la)  +|\D \pi_\la|^2 (A^j\de_j\pi_\la))\cdot V|  \nn \\
&& +C(\log\la)^{\fr12}(\la^{-1} + \la(\al-1) )\|V\|_z \nn
\eeqnas	
Once again by direct calculation we have 
$$\Dl(A^j\de_j\pi_\la)+A^j\de_j\pi_\la |\D \pi_\la|^2 = - \pi_\la A^j\de_j |\D \pi_\la|^2 $$
and using again that $V\cdot z =0$ a.e.: 
$$|A^j\de_j |\D \pi_\la|^2 \pi_\la \cdot V|  \leq \fr{\la^2|x|^2}{(1+\la^2|x|^2)^2}\rho_z |V|.$$
We leave the final details to the reader since they are analogous to the previous case. 
\end{proof}

\subsection{Difference estimates for $\ed E_\al$ and $\ed^2 E_\al$ near $\ZZ$}
 
Here we will consider $u\in W^{1,\infty}(\Sigma, \Sp^2)$ which are close to a specific region of $\ZZ$ in the sense that there is some $z\in \ZZ$ and $
\Lambda<\infty$ so that $$\|u-z\|_{L^{\infty}(\Sigma)}\leq \Lambda \quad \text{and} \quad  \max\{(2+|\D u|^2)^{\al-1},(2+|\D z|^2)^{\al-1}\}  \leq 1+\Lambda.$$ 
We will apply the below results in the setting that $u$ is $\al$-harmonic and satisfies the hypotheses of the main Theorems, so in particular $\Lambda$ can be thought of as small. 

For future reference we first note, setting $u_t=P(z+tw)=\frac{z+tw}{|z+tw|}$ and for $A,B \in \Gamma^2(z)$, $C,D\in \Gamma^2(u_t)$: 
\beqna
	\ed^2 E_\al(z)[A,B] &-& \ed^2E_\al(u_t)[C,D] = \nn\\
	&&\al \int_\Sigma (2+|\D u_t|^2)^{\al-1}(\D A \cdot (\D B - \D D) + (\D A - \D C)\cdot \D D) \nn \\
	&&- \al \int_\Sigma (2+|\D u_t|^2)^{\al-1}|\D z|^2(A \cdot (B - D) + ( A -  C)\cdot  D) \nn\\
	&& + \al \int_\Sigma ( (2+|\D z|^2)^{\al-1}-(2+|\D u_t|^2)^{\al-1})\left[\D A \cdot \D B - |\D z|^2A\cdot  B\right] \nn\\ 	
	&& + \al \int_\Sigma (2+|\D u_t|^2)^{\al-1}(|\D u_t|^2-|\D z|^2) C \cdot  D\nn \\ 
	&& +2 \al(\al-1) \int_\Sigma (2+|\D z|^2)^{\al-1}\frac{(\D z \cdot \D A)(\D z \cdot \D B)}{2+|\D z|^2}\nn \\
	&& + 2\al(\al-1) \int_\Sigma (2+|\D u_t|^2)^{\al-1}\fr{(\D u_t \cdot \D C)(\D u_t \cdot \D D)}{2+|\D u_t|^2}. \label{eq:2nddiff}
\eeqna 
 
\begin{lemma}\label{lem:e3} Let $\Lambda>0$, $u\in W^{1,\infty}(\Sigma,\Sp^2)$ and $z\in \ZZ$. Suppose 
$\max\{(2+|\D u|^2)^{\al-1},(2+|\D z|^2)^{\al-1}\} \leq 1+\Lambda$ and $\|u-z\|_{L^{\infty}}\leq \tfrac12$.  Setting $w:=u-z$, $u_t=P(z+tw)$, $W_t = \ed P(z+tw)[w]$ and $W=W_0$, there exists $C=C(\Sigma)<\infty$ so that \begin{equation*}
	|\ed^2 E_\al (z)[W,W]-\ed^2 E_\al(u_t)[W_t,W_t] |  \leq C(1+\Lambda)(\Lambda + \|w\|_{L^{\infty}} + (\al-1))\|w\|_z^2,
\end{equation*}
and 
$$|\ed E_\al(u_t)[\ed P(u_t)[\de_t W_t]]|\leq C(1+\Lambda)\|w\|_{L^{\infty}}\|w\|_z^2.$$

\end{lemma}

\begin{proof}
First that since $w^{\bot_z} = (w\cdot z) z = -\tfrac12|w|^2z $ we have $W=w+\tfrac12|w|^2 z$ giving, using very rough estimates 
$$\tfrac12|w|\leq |W| \leq |w|, \quad |\D W|\leq 2(|\D w| + |w|^2 |\D z|),\quad \text{and}\quad |\D w|\leq 2(|\D W| +|W|^2|\D z|) .$$ Thus there is some $C$ so that 
\beq\label{eq:Ww}
\|W\|_{L^\infty}\simeq \|w\|_{L^\infty}\quad\text{and}\quad  \|W\|_z \simeq \|w\|_z.  \eeq 

Straight from the definition: 
\beqa\label{eq:WWtdiff}
|W_t-W| &=|\ed P(z+tw)[w]-\ed P(z)[w]| \leq C|W|^2, \quad \text{and}\\
|\D(W_t-W)|&\leq C(|W||\D W|+ |W|^2 |\D z|) 
\eeqa   
so in particular using \eqref{eq:Ww}  
\beq\label{eq:WWt}
\|W-W_t\|_z\leq C\|w\|_{L^{\infty}}\|w\|_z, \qquad \|W_t\|_z \simeq \|w\|_z \quad\text{and}\quad \|W_t\|_{L^\infty}\simeq \|w\|_{L^\infty} .
\eeq 

We also have $\D u_t = \D z(1+O(t|w|)) + t\D w(1+O(|w|))$ which gives  
\beqa \label{eq:zut}
|\D u_t|\leq C(|\D z| + |\D w|),&\qquad  |\D u_t-\D z|\leq Ct|\D w| + Ct|w||\D z|, \\
 ||\D z|^2 -|\D u_t|^2|&\leq C(|\D w|^2 + |\D z||\D w|+ |\D z|^2|w|)
\eeqa
and 
\beq\label{eq:ut}
|\D u_t|\leq (1-t)|\D z|(1+O(|w|)) + t|\D u|(1+O(|w|)). 
\eeq
By the hypotheses on $u$ and $z$, using \eqref{eq:ut} we have  
\beq\label{eq:utal}
1\leq (2+|\D u_t|^2)^{\al-1} \leq 1+C\Lambda\quad \text{for some $C<\infty$}
\eeq and then 
$$|(2+|\D u_t|^2)^{\al-1} - (2+|\D z|^2)^{\al-1}|\leq C\Lambda.$$
Combining this with \eqref{eq:Ww}-\eqref{eq:zut} and \eqref{eq:2nddiff} give for $A=B=W$ and $C=D=W_t$, recalling that $|\D z|\leq C\rho_z$ on $\Sigma$:
\beqna
	|\ed^2 E_\al(z)[W,W] &-& \ed^2E_\al(u_t)[W_t,W_t]| \leq  \nn\\
	&&  C(1+\Lambda) \int_\Sigma (|\D W| + |\D W_t|)|\D (W-W_t)| + \rho_z^2(|W| + |W_t|)|W-W_t| \nn \\
	&&+  C\Lambda   \int_\Sigma |\D W|^2 + \rho_z^2|W|^2 + C(1+\Lambda)\int_\Sigma (|\D w|^2 + \rho_z|\D w|+ \rho_z^2|w|)|W_t|^2 \nn\\ 	 
	&& +C(1+\Lambda)(\al-1) \int_\Sigma |\D W|^2 + |\D W_t|^2\nn \\
	& & \leq C(1+\Lambda)(\|w\|_{L^{\infty}}\|w\|^2_z + \Lambda\|w\|_z^2 + (\al-1)\|w\|_z^2). \nn \eeqna 
Now, 
\beqas
\ed P(u_t)[\de_t W_t]&= \ed P(u_t)(\ed^2P(z+tw)[w,w]) \\
&= \ed P(z)[\ed^2 P(z)[w,w]]+ \ed P(z)[(\ed^2P(z+tw) - \ed P^2(z))[w,w]]\\
&\quad  + (\ed P(u_t)-\ed P (z))[\ed^2 P(z+tw)[w,w]]. 
\eeqas

Thus by \eqref{B2}, using that $A(z)(V,W) = -z (V\cdot W)$ for the round sphere, this gives 
\beqas
\ed P(u_t)[\de_t W_t]&= |w|^2W +\ed P(z)[(\ed^2P(z+tw) - \ed P^2(z))[w,w]] \\
&\quad + (\ed P(u_t)-\ed P (z))[\ed^2 P(z+tw)[w,w]]
\eeqas 
which in particular yields 
$$|\D (\ed P(u_t)[\de_t W_t])|\leq C( |\D W||W|^2+ |\D z||W|^3).$$
This and \eqref{eq:zut} give  
\beqas 
|\ed E_\al(u_t)[\ed P(u_t)[\de_t W_t]]|&\leq C(1+\Lambda)\int_\Sigma |\D u_t||\D W||W|^2 + |\D u_t||\D z||W|^3 \\
&\leq C(1+\Lambda)\left(\int_\Sigma |\D z||\D W||W|^2 + |\D w||\D W||W|^2 + |\D z|^2|W|^3 \right.\\
&\quad + \left.\int_\Sigma |\D w||\D z||W|^3\right) \\
&\leq C(1+\Lambda)\|w\|_{L^{\infty}}\|w\|_z^2
\eeqas
as required.  
\end{proof}

\begin{lemma}\label{lem:e4} Let $u\in W^{1,\infty}(\Sigma,\Sp^2)$ and $z\in \ZZ$. Suppose 
$\max\{(2+|\D u|^2)^{\al-1},(2+|\D z|^2)^{\al-1}\} \leq 1+\Lambda$ and $\|u-z\|_{L^{\infty}}\leq \Lambda$ for some $\Lambda <\infty$. Set $w:=u-z$, $u_t=P(z+tw)$ and $W_t = \ed P(z+tw)[w]$ and given $T\in \Gamma^2(z)$ let $T_t = \ed P(u_t)[T]$. Then there exists $C=C(\Sigma)<\infty$ so that:  

\begin{enumerate}
\item When $T^\la=\de_\la z$
\beqas
	|\ed^2E_\al(z)[T^\la,W]-\ed^2 E_\al(u_t)[T^\la_t,W_t]|&\leq C\la^{-1}(1+\Lambda)\left[  \|w\|_z^2 +(\al-1)\|w\|_z	\right]
	\eeqas and 
\begin{equation*}
	|\ed E_\al (u_t)[\ed P(u_t)\de_t T^\la_t]| \leq C\la^{-1}(1+\Lambda)\|w\|_z^2. 
\end{equation*}

\item 	When $T^A=\D_A z$,\beqas
	|\ed^2E_\al(z)[T^A,W]-\ed^2 E_\al(u_t)[T^A_t,W_t]|&\leq C\la(1+\Lambda) (\|w\|^2_z + (\al-1)\|w\|_z)	\eeqas and  \begin{equation*}
	|\ed E_\al (u_t)[\ed P(u_t)\de_t T^A_t]| \leq C\la (1+\Lambda) \|w\|_z^2.
	\end{equation*}
\end{enumerate}

\end{lemma}

\begin{proof}
	
Given $0\leq x\leq y\in \R$ assume that $(2+y^2)^{\al-1} \leq 1+\Lambda$ then the mean value formula straightforwardly gives the existence of some $\xi\in [x,y]$ so that:
\beqna\label{eq:mva}
(2+y^2)^{\al-1} - (2+x^2)^{\al-1}&=&(\al-1)\fr{2\xi}{(2+\xi^2)}(2+\xi^2)^{\al-1}(y-x)\nn\\
&\leq & 4(1+\Lambda)(\al-1)\frac{1}{\sqrt{2} + \xi} (y-x).
\eeqna Combining this with \eqref{eq:zut} gives 
\begin{equation}\label{eq:aldif}
|(2+|\D z|^2)^{\al-1} - (2+|\D u_t|^2)^{\al-1}| \leq C(1+\Lambda)(\al-1)\frac{|\D w| + |w||\D z|}{\sqrt{2} + |\D z|} \quad\text{if} \quad |\D u|\geq \fr12 |\D z|	\end{equation}
whilst we trivially note for later that \beq\label{eq:else}
|\D z|\leq 2|\D w| \quad \text{if} \quad  |\D u|< \fr12 |\D z|.
\eeq
Directly from \eqref{eq:mva} and \eqref{eq:zut} we also see that 
\beq\label{eq:di}
|(2+|\D z|^2)^{\al-1} - (2+|\D u_t|^2)^{\al-1}| \leq C(1+\Lambda)(\al-1)(|\D w| + |w||\D z|).
\eeq
Using $|\D z|\simeq \rho_z$ on $B_\iota(a)$ together with  \eqref{eq:2nddiff}, \eqref{eq:zut}, \eqref{eq:utal} and  \eqref{eq:aldif}-\eqref{eq:di} gives, with $U_1 = \{|\D u|\geq \fr12|\D z|\}$ and $U_2 = \Sigma\sm U_1$  
\beqna
	|\ed^2 E_\al(z)[T^\la,W] &-& \ed^2E_\al(u_t)[T_t^\la,W_t]| \leq  \nn\\
	&&C(1+\Lambda)\int_\Sigma |\D T^\la||\D (W-W_t)|+ |\D (T^\la -T_t^\la)| |\D W_t| \nn \\
	&&+C(1+\Lambda)  \int_\Sigma \rho_z^2(|T^\la||W-W_t|  + |W_t||T^\la-T_t^\la|) \nn\\
	&& + C(1+\Lambda)(\al-1) \int_{B_\iota(a)\cap U_1} \frac{|\D w| + |w|\rho_z}{\sqrt{2} + \rho_z}\left[|\D T^\la ||\D W| + \rho_z^2|T^\la||W|\right] \nn\\ 	
	&& + C\Lambda  \int_{B_\iota(a)\cap U_2} \left[|\D T^\la ||\D W| + \rho_z^2|T^\la||W|\right] \nn\\
	&& +C(1+\Lambda)(\al-1)\int_{\Sigma \sm B_\iota(a)}(|\D w| + |w|\rho_z)\left[|\D T^\la ||\D W| + \rho_z^2|T^\la||W|\right]\nn\\
	&& +C(1+\Lambda)  \int_\Sigma (|\D w|^2 + |\D z||\D w|+ |\D z|^2|w|) |T^\la_t||W_t| \nn\\ 
	&& +C(1+\Lambda)(\al-1) \int_\Sigma |\D T^\la||\D W| +|\D T_t^\la||\D W_t|.\nn 
\eeqna 
We collect the following readily-checked facts concerning $T_t$ which follow from the definition: 
\beq\label{eq:ttt}
|T_t-T| \leq C |W||T| \quad |\D(T_t-T)|\leq C(|W||\D T|+|W||\D z||T| + |\D W||T|). 
\eeq 

Now for $T=T^\la = \de_\la z$ we have $|T^\la| \leq C \la^{-1}$ and by \eqref{eq:ttt} we have the same bound for $T^\la_t$.   
Also, 
\beqs
|\D T^\la|\leq C\la^{-1}\rho_z \quad \text{and by \eqref{eq:ttt}} \quad |\D T^\la_t|\leq  C(\la^{-1}\rho_z + \la^{-1}\D W).
\eeqs     Note that both $T^\la$ and $\D T^\la$ are bounded by $\la^{-2}$ on $\Sigma \sm B_{\iota}(a)$. The above now becomes, using also \eqref{eq:WWtdiff}, \eqref{eq:else} and again with $U_1 = \{|\D u|\geq \fr12|\D z|\}$ and $U_2 = \Sigma\sm U_1$: 
\beqas
|\ed^2E_\al(z)[T^\la,W]&-\ed^2 E_\al(u_t)[T^\la_t,W_t]|\leq \\
&  C\la^{-1}(1+\Lambda)\int_\Sigma \rho_z(|W||\D W| + |W|^2\rho_z) +\rho_z|W||\D W_t| + |\D W||\D W_t| \\
& +   C\la^{-1}(1+\Lambda) \int_\Sigma \rho_z^2(|W|^2+|W_t||W|)\\
&+  C\la^{-1}(1+\Lambda)(\al-1)\int_{B_\iota(a)\cap U_1} (|\D w| + |w|\rho_z) ( |\D W|+\rho_z|W|)  \\ 
&+ C\la^{-1}\Lambda\int_{B_\iota(a)\cap U_2} |\D W||\D w| + |W|\rho_z |\D w| \\
& + C\la^{-2}(1+\Lambda)(\al-1)\int_{\Sigma \sm B_\iota(a)}(|\D w| + |w|\rho_z)\left[|\D W| + \rho_z|W|\right]\\
& +C\la^{-1}(1+\Lambda)(\al-1)\int_\Sigma\rho_z(|\D W| + |\D W_t|) + |\D W||\D W_t|
\eeqas
which readily gives the required estimate by \eqref{eq:Ww} and \eqref{eq:WWt}. 

By definition $\de_t T_t = \ed^2 P(u_t)[T_t, W_t]$, so by \eqref{B1} and \eqref{B2}, $\ed P(z)((\ed^2 P (z)[W ,T])) =0$ giving 
\beqas
\ed P(u_t)\de_t T_t&=  (\ed P(u_t) -\ed P(z))[\ed^2 P(z)[T,W]] + \ed P(u_t)[(\ed^2 P(u_t) - \ed^2 P(z))[T,W]]\\
&\quad   + \ed P(u_t)[\ed^2 P(u_t) [T_t-T, W]] + \ed P(u_t)[\ed^2P(u_t)[T_t, W_t-W]].  
\eeqas
This means in particular using \eqref{eq:WWtdiff} -- \eqref{eq:zut} and \eqref{eq:ttt}   
$$|\D \ed P(u_t)\de_t T_t| \leq C( |\D z||W|^2|T|+|W||\D W||T| + |\D T||W|^2)$$
and for $T=T^\la$
$$\ed E_\al(u_t)[\ed P(u_t)\de_t T^\la_t]\leq C\la^{-1}\int_\Sigma (\rho_z + |\D w|)(\rho_z|W|^2 + |W||\D W| + \rho_z|W|^2)\leq C\la^{-1}\|w\|_z^2. $$

\
When $T=T^A$ the proof follows exactly the same lines expect that we use $|T^A|\leq C\la$ and $|\D T^A|\leq C\la\rho_z$, whilst both $T^A$ and $\D T^A$ are bounded by $\la^{-1}$ on $\Sigma \sm B_{\iota}(a)$. These give precisely the same estimates with a factor of $\la^2$ the only final difference.
\end{proof}


\appendix 

\section{Initial expansions}\label{app:A}

In $U_r(a)=F_a^{-1}(\DD_r)$, $\tilde{z}_{\lambda,a} = \pi_\la(F_a(p)) + j_{\lambda,a} (F_a(p))$ meaning that in $F_a$-coordinates $x\in \DD_r$ we have
\beqna\label{eq:zr}
z(x)=z_{\lambda,a}(x) &=& P(\pi_\la (x) + j_{\lambda,a}(x)) = \pi_\la + \int_0^1 \pl{}{t} P(\pi_\la + tj_{\lambda,a}) \id t \nn \\
&=& \pi_\la + \ed P(\pi_\la)[j_{\lambda,a}] + \int_0^1 (\ed P(\pi_\la + tj_{\lambda,a})[j_{\lambda,a}] - \ed P(\pi_\la)[j_{\lambda,a}]) \id t \nn \\
&=& \pi_\la + j_{\lambda,a}^\top + K_{\lambda,a}  
\eeqna
where $j_{\lambda,a}^\top  = \ed P(\pi_\la)[j_{\lambda,a}]= j_{\lambda,a} - (j_{\lambda,a}  \cdot \pi_{\lambda} ) \pi_\la$. 
Mostly we will drop the explicit ${\la,a}$ dependencies from $z$, $j$ and $K$, writing simply $z(x)=\pi_\la + j^\top + K$.

We will frequently use standard properties of $\pi_\la(x)=\left(\frac{2\la x}{1+\la^2|x|^2} , \frac{1-\la^2|x|^2}{1+\la^2|x|^2}\right) $:

\beqa\label{eq:pop}
 -\Dl\pi_\la = |\D \pi_\la|^2 \pi_\lambda,\qquad \qquad &\pi_\la\cdot \de_i \pi_\la = 0,\quad i=1,2 \\
 \de_1 \pi_\la \cdot \de_2 \pi_\la =0,\quad  |\de_1 \pi_\la|^2 =|\de_2 \pi_\la|^2,   \qquad  &\de_\la \pi_\la = \la^{-1} x\cdot \D \pi_\la (x), \\
   |\D \pi_\la|^2=\frac{8\la^2}{(1+\la^2|x|^2)^2}, \qquad   \qquad &|\D^2 \pi_\la|= O\left(\frac{\la^2}{(1+\la^2|x|^2)^\fr32}\right).
\eeqa
On $U_r(a)\cong \DD_r$ we have the following rough estimates as can be checked directly: 
 From the definition of $j=j_{\la,a}$ we have $|\D j|= O(\la^{-1})$ and since $j(0) = 0$ we have 
\beq\label{eq:rough1}
|j| =O (|x|\la^{-1}) \quad \text{giving} \quad |\D K|=O(|x|\la^{-2}), \quad \text{and} \quad |\D j^\top|=O(\la^{-1})
\eeq
after a short calculation. 

Using \eqref{eq:pop} and the properties of $P$ along $\mathbb{S}^2$ we also have 
\beqa\label{eq:rough1a}
|\D^2 j^\top|=O\left(\la^{-1} + (1+\la^{2}|x|^2)^{-1}\right), \quad |\D^2 K |=O(\la^{-2}) \quad \text{and}\quad 
 |(\Dl K)^\top|=O\left(\frac{\la^{-1}|x|}{1+\la^2|x|^2}\right). 
\eeqa
Recall that the third component of $j$ is zero, and $\pi_\la^{1,2}=\fr{2\la x}{1+\la^2|x|^2}$, which gives \beq\label{eq:jpi} j\cdot\pi_\la = O\left(\frac{|x|^2}{1+\la^2|x|^2}\right)\quad \text{and} \quad |\D (j \cdot \pi_\la)|= O\left(\frac{|x|}{1+\la^2|x|^2}\right).\eeq
\subsection{Expansions involving $\de_\la z$}

After further estimates one can check that
\beq\label{eq:rough2}
|\de_\la j|=O(|x|\la^{-2})\quad \text{giving} \quad|\de_\la j^\top|=O(|x|\la^{-2}), 
\quad \text{and} \quad |\de_\la K|=O(|x|^2\la^{-3}).
\eeq

For our purposes later we note that
\begin{lemma}\label{lem:rough3}
\beq\label{eq:rough3}
 \Dl j^\top \cdot  \de_\la j^\top=O\left(\frac{\la^{-1}|x|^2}{(1+\la^2|x|^2)^2}\right).
\eeq	
\end{lemma}
\begin{proof}
Since $j$ is harmonic we have 
\beqna\label{eq:Lapjtop}
\Dl j^\top &=& -\Dl [ (j\cdot \pi_\la) \pi_\la ]\nn \\ 
&=& -(\Dl (j\cdot \pi_\la) ) \pi_\la - (j\cdot \pi_\la) \Dl \pi_\la - 2 \D (j\cdot \pi_\la)\cdot \D \pi_\la  \nn\\
&=& -(2\D j\cdot \D \pi_\la  -2 |\D \pi_\la|^2j\cdot \pi_\la)\pi_\la  - 2 \D (j\cdot \pi_\la)\cdot \D \pi_\la.  
\eeqna	

We also have
\beqas
\de_\la j^\top &= \de_\la (j - (j\cdot \pi_\la)\pi_\la) \\
&= [(\de_\la j)^\top  - (j\cdot \pi_\la) \de_\la \pi_\la] - (j\cdot \de_\la\pi_\la)  \pi_\la.
\eeqas
These together give 
\beqas
\Dl j_{\lambda,a}^\top \cdot  \de_\la j_{\lambda,a}^\top &= -(2\D j\cdot \D \pi_\la  -2 |\D \pi_\la|^2j\cdot \pi_\la)(j\cdot \de_\la\pi_\la) \\
&\quad - 2 \D (j\cdot \pi_\la)\cdot \D \pi_\la \cdot [(\de_\la j)^\top  - (j\cdot \pi_\la) \de_\la \pi_\la].
\eeqas

Using $\de_\la \pi_\la = \la^{-1} x\cdot \D \pi_\la$ and $\D j\cdot \D \pi_\la = \frac{O(1)}{(1+\la^2|x|^2)}$ and the aforementioned estimates we see therefore that 
\beqa 
|\Dl j_{\lambda,a}^\top \cdot  \de_\la j_{\lambda,a}^\top| &\leq  |(2\D j\cdot \D \pi_\la  -2 |\D \pi_\la|^2j\cdot \pi_\la)(j\cdot \de_\la\pi_\la) \\
& \quad  + 2 \D (j\cdot \pi_\la)\cdot \D \pi_\la \cdot [(\de_\la j)^\top  - (j\cdot \pi_\la) \de_\la \pi_\la]|\nn \\
&\leq C\la^{-3}\frac{\la^2|x|^2}{(1+\la^2|x|^2)^2}.
\eeqa

\end{proof}
We finally have, on $U_r(a)$: 
\beq\label{eq:rough}
|\D z_{\lambda,a}| \simeq\frac{\la}{1+\la^2|x|^2}, \qquad |\Dl z_{\lambda,a}| \simeq\frac{\la^2}{(1+\la^2|x|^2)^2},\qquad|\D^2 z|\simeq\frac{\la^2}{(1+\la^2|x|^2)^\fr32},
\eeq
along with 
\beq\label{eq:rouhgla}
|\de_\la z|\simeq\la^{-1}\frac{\la|x|}{1+\la^2|x|^2} \quad \text{and} \quad |\D \de_\la z|\simeq \frac{1}{(1+\la^2|x|^2)}.
\eeq

On the other hand, in $\Sigma\sm U_r(a)$ it is straightforward to check, using \eqref{eq:ztilde}, and \eqref{eq:zloc} that we have the following estimates on $z_{\lambda,a}$:   
\beq\label{eq:rough_la_away}
|\D z_{\lambda,a}| = O(\la^{-1}) \quad \text{and}\quad  \left|\pl{z_{\lambda,a}}{\la}\right|, \, \left|\Dl z_{\lambda,a}\right|, \, \left|\D |\D z_{\lambda,a}|^2\right|, \, \left|\D \de_\la z\right| = O(\la^{-2}).  
\eeq

\subsection{Expansions involving $\D_A z$ when $\gamma\geq 2$}
Let $A\in T_a \Sigma\cong_{(F_a)_\ast} T_0\mathbb{D}$ and for $0\leq s\ll 1$, $a_s: = Exp_a^{\Sigma} (sA)$ and we wish to compute 
$$\D_A z_{\la ,a}(p) = \pl{}{s} z_{\la,a_s}(p) \vlinesub{s=0}.$$ 

Given our choice of local coordinates $F_a$ at $a$ we choose $F_{a_s} : = (\tau^A_{s})^{-1} \circ F_a$ where $\tau^A_s$ is a Hyperbolic translation along the geodesic $Exp^{\mathbb{D}}_0(sA)$ by distance $s|A|_{g(0)}$. Specifically $Exp^{\mathbb{D}}_0(sA)=2A\tanh(s/2):=\gamma(s)$ and $\tau^A_s$ is given explicitly by (see e.g. \cite{Rat06})
$$(\tau_s^A)^{-1}(x)=\frac{(1-|\gamma(s)|^2)x - (|x|^2 - 2x\cdot \gamma(s) + 1)\gamma(s)}{|\gamma(s)|^2|x|^2-2x\cdot \gamma(s) + 1}$$
where all norms and inner-products appearing above are Euclidean.  In particular 
\begin{equation}\label{dtau}
	\pl{}{s}(\tau_s^A)^{-1}(x)\vlinesub{s=0}= -A - A|x|^2 + 2(x\cdot A)x. 
\end{equation}
Thus for $p$ close to $a$ we have 
\beq\label{eq:daz}
\D_A z(p) = \D_A \pi_\la (p) + \D_A (j^a)^\top(p) + \D_A K(p)  
\eeq
where we abuse notation by writing $$\D_A \pi_\la = \pl{}{s}\vlinesub{s=0}\pi_\la(F_{a_s}(p)),$$ $$\D_A (j_\la^{a})^\top = \pl{}{s}\vlinesub{s=0}(j_\la ^{a_s}(p) -(j_\la ^{a_s}(p) \cdot \pi_\la(F_{a_s}(p)))\pi_\la(F_{a_s}(p)))$$ etc. 

In particular by $\eqref{dtau}$ we see that for $F_a$ coordinates $x\in \DD_r$, 
\begin{eqnarray}\label{eq:deApila}
	\D_A \pi_\la 
	= -A^i\de_i \pi_\la -|x|^2 A^i\de_i \pi_\la + 2(A\cdot x) x^i\de_i \pi_\lambda  ,
\end{eqnarray}
noting that we have $\D_A \pi_\la \cdot \pi_\la \equiv 0$.

For $A$ a unit vector we collect the following rough estimates on $\DD_r$ which can be checked by direct calculation - see also \ref{App:deA}. First of all $|\D_A j_{\la,a}|\leq C\la^{-1}$ which yields that 
\beq\label{eq:roughA}
|\D_A j^\top|=O(\la^{-1}),\quad |\D_A K_{\la,a}| =O(|x|\la^{-2}) \quad \text{and} \quad |\D_A j_{\la,a}^\top \cdot \Dl j_{\la,a}^\top| =O\left(\frac{|x|}{(1+\la^2|x|^2)^2}\right)
\eeq
where the final estimate above follows by suitably adapting the proof of Lemma \ref{lem:rough3} where the only difference is that $\D_A j \sim \la|x|^{-1} \de_\la j$ and $\D_A \pi_\la \sim \la |x|^{-1}\de_\la\pi_\la$.  

By appendix \ref{App:deA} we also have $|\D \D_A j|\leq C\la^{-1}$ which gives 
\beq
|\D \D_A j^\top| =O\left(\la^{-1} + (1+\la^2|x|^2)^{-1}\right)  \quad \text{and} \quad |\D \D_A K|= O(\la^{-2}).
\eeq 
Thus we have  
\beq\label{eq:roughA1}
 |\D_A z|\simeq \frac{\la}{1+\la^2|x|^2} \quad \text{and}\quad |\D \D_A z| \simeq \frac{\la^2}{(1+\la^2|x|^2)^\fr32}     \eeq
on $U_r(a)$. 

On the other hand, in $\Sigma\sm U_r(a)$ it is straightforward to check,
\beq\label{eq:roughAaway}
|\D_A z_{\la, a}| \,\,, |\D \D_A z_{\la,a}| = O(\la^{-1}). 
\eeq

\subsubsection{Computing $\D_A j$}\label{App:deA} 
Here we expand on the computations for $\D_A j$ - we note that the precise form of the derivation below is not crucial for our needs, but we nevertheless give full details for the sake of the reader.    

We have 
$$\pl{}{s}\vlinesub{s=0}\tau^A_s(x) = A + A|x|^2 - 2(x\cdot A)x$$
noting for later that,
\begin{eqnarray}\label{eq:dlog}
\pl{}{s}\vlinesub{s=0}\D_y \log 
|\tau_s^{-A}(x)-y|\vlinesub{y=0} &=& -A^j\de_{x^j}\D_y \log|x-y|\vlinesub{y=0} + A	\nn\\
&=& A^j\de_{y^j}\D_y \log|x-y|\vlinesub{y=0} + A
\end{eqnarray} and   
\begin{eqnarray}\label{eq:A20}
	\fr{\la}{2}(j_{\la,a_s})^{1,2}(p) &=& \D_y J_{a_s} (F_{a_s}(p),y)\vlinesub{y=0} - \D_y J_{a_s}(F_{a_s}(a_s),y)\vlinesub{y=0}\nn \\
	&=& \left( \D_y (G(p, F_{a_s}^{-1}(y)))  + \D_y \log|F_{a_s}(p) - y|\right)\vlinesub{y=0}\nn\\
	&& - \left(\D_y (G(F_{a_s}^{-1}(0), F_{a_s}^{-1}(y))) + \D_y\log |y|\right)\vlinesub{y=0}. \end{eqnarray}

By definition we have $F_{a_s}^{-1} = F_a^{-1}\circ \tau^A_s$ and in the below we set $x=F_a(p)$ so that $F_{a_s}(p) = \tau_s^{-A}(x)$.  
Thus we have, by \eqref{eq:dlog} and \eqref{eq:A20}:
\begin{eqnarray*}
	\pl{}{s}\vlinesub{s=0}\fr{\la}{2}(j_{\la,a_s})^{1,2}(p)&=& \left(A^j\de_{y^j}\D_y (G (p,F_{a}^{-1}(y))) + A^j\de_{y^j}\D_y \log|x-y|\right)\vlinesub{y=0} + A \\
	&& -\left(A^j\de_{x^j}\D_y (G(F_a^{-1}(x),F_a^{-1}(y))) + A^j\de_{y^j} \D_y (G(F_a^{-1}(x),F_a^{-1}(y)))\right)\vlinesub{x,y=0} \\
	&=& A^j\de_{y^j}\D_y J_a (x,0) + A  - A^j\de_{x^j}\D_y J_a(0,0)- A^j\de_{y^j} \D_y J_a(0,0) \\
	&& + \left(A^j\de_{x^j}\D_y \log|x-y| + A^j\de_{y^j} \D_y \log|x-y|\right)\vlinesub{x,y=0}   \\
	&=& A^j\de_{y^j}\D_y J_a (x,0)- A^j\de_{x^j}\D_y J_a(0,0)- A^j\de_{y^j} \D_y J_a(0,0) + A. 
\end{eqnarray*}
Thus in particular 
\begin{equation}
\D_A j = \fr{2}{\la}(A^j\de_{y^j}\D J_a(x,0), 0) + \fr{1}{\la}(\mathcal{C},0) 	
\end{equation}
for a constant vector $\mathcal{C}\in \R^2$.

\section{Derivatives of $\mathcal{J}$}\label{App:J}

We work with the Green's function which solves 
\begin{equation}\label{eq:Grn}
 -\Dl_p G(p,a) = 2\pi\dl_a - \frac{2\pi}{\Vol_g(\Sigma)}.	
 \end{equation}
If $x,y$ are local coordinates near $a\in \Sigma$ as in \eqref{eq:xycoords}, then for $z=x^1 + ix^2$, $\zeta = y^1+iy^2$ 
$$\mathcal{J}(a):= (\de_{x^1}\de_{y^1} J_a(0,0) + 
 \de_{x^2}\de_{y^2} J_a(0,0)) = 4 \lim_{x\to 0} {\rm Re}\left(\ppl{G_a}{\bar\zeta}{z}(x,0)\right)$$
 where, from \cite[Proposition 6.2, Remark 6.3]{MRS23}, in any holomorphic coordinates $(z,\zeta)$ we have 
 $$\ppl{G}{\bar\zeta}{z}(z,\zeta) \ed z \otimes \ed \ov\zeta= -\pi \sum_j \phi_j(z)  \otimes \overline{\phi_j(\zeta)}$$ and $\{\phi_j\}_j$ is an $L^2$-orthonormal basis of holomorphic one-forms on $\Sigma$. In particular when $\Sigma$ has genus $\gamma$ we have
  $$\mathcal{J}(a)=-2\pi c_\gamma\sum_j |\phi_j(a)|_g^2$$
  where $c_1 = 1$ and when $\gamma\geq 2$, $c_\gamma = 4$. 
\begin{lemma}\label{lem:deJ}
With the set-up as above we have 
$$-\fr12\D_A \mathcal{J}(a) = A^i\de_{x^i} [\de_{x^1}\de_{y^1} J_a(0,0) + 
 \de_{x^2}\de_{y^2} J_a(0,0)].$$
\end{lemma}
\begin{proof}
We first note that   
\begin{eqnarray*}
-\fr12\D_A \mathcal{J}(a) &=& \pi c_\gamma  \D_A^\Sigma \sum_j \lan \phi_j,\ov \phi_j\ra_g   = \pi c_\gamma \sum_j \lan \D_A^\Sigma \phi_j, \bar\phi_j \ra_g + \lan\phi_j, \ov{\D_A^\Sigma\phi_j}\ra_g \\&=& 2\pi c_\gamma {\rm Re}\left(\sum_j \lan \D_A^\Sigma \phi_j(a), \bar\phi_j(a)\ra_g \right)
\end{eqnarray*}
where $\lan, \ra_g$ and $\D_A^\Sigma$ are the metric and connection on $\Sigma$ extended complex linearly to $T^{\C}\Sigma$ respectively.  	

In our local coordinates $x$ we write $\phi_j  = \ti\phi_j \ed z$ and $A=A^i\de_{x^i}$. Since $\D^\Sigma\ed z (0) = 0$ we have that $\D^\Sigma_A \phi_j (0) = (A^i\de_{x^i} \ti\phi_j(0))\ed z$. Noting also that $A^i\de_{x^i}\ppl{G_a}{\bar\zeta}{z}(x,0)=-\pi(A^i\de_{x^i}\ti\phi_j(x))\ov{\ti\phi_j}(0)$ and $|\ed z(0)|^2_g = 2c_\gamma^{-1}$, gives   
\begin{eqnarray*}
-\fr12\D_A \mathcal{J}(a)  &=& 4\pi {\rm Re}\left(\sum_j (A^i\de_{x^i} \ti\phi_j(0)) \ov{\ti\phi_j}(0)\right) = 4\lim_{x\to 0}{\rm Re}\left(A^i\de_{x^i}\ppl{G_a}{\bar\zeta}{z}(x,0) \right)	\\&=& A^i\de_{x^i} [\de_{x^1}\de_{y^1}  + 
 \de_{x^2}\de_{y^2} ]J_a (0,0),\end{eqnarray*}
 since for $x\neq y$, by \eqref{eq:locG-deriv} 
 $\ppl{G_a}{\bar\zeta}{z}(x,y) = \fr14(\de_{x^1}-i\de_{x^2})(\de_{y^1}+i\de_{y^2})J_a (x,y).$
 \end{proof}

\section{First and second variation of $E_\al$} 
We briefly derive the first and second variations of $E_\al$ along maps $u\in W^{1,2\al}(\Sigma,N)$ assuming $N\emb \R^d$ is closed and isometrically embedded and $\Sigma$ is any closed Riemannian manifold.  

Let $\Om_{\nu}(N)$ denote a small $\nu$-neighbourhood of $N$ in $\R^d$ where $\nu>0$ is small enough so that $P:\Om_\dl(N)\to N$, the nearest point projection, is well-defined and smooth.  We note the following easy to check facts concerning the behaviour of $P$ on $N$. Specifically given \emph{any} $X,Y\in \R^d$ and $k\in N$: 
 \begin{eqnarray}
\ed P(k)[X] 	&=& X^{\top} \qquad \text{where $X^\top$ is the projection onto $T_k N$} \label{B1} \\
 	\ed^2 P(k)[X,Y] &=& A(k)(X^\top, Y^\top) + \lan A(k)(\cdot, X^\top), Y^\bot\ra^\flat + \lan A(k)(\cdot, Y^\top), X^\bot\ra^\flat \label{B2}
 \end{eqnarray}  
 where of course $X^\bot = X-X^\top$ is the projection onto $\mathcal{V}_k N$ and $A(k)(X^\top,Y^\top): = \left(\D^{\R^d}_{X^\top} Y^\top \right)^\bot$. For a round sphere $A(k)(X^\top,Y^\top)=-k \lan X^\top,Y^\top\ra$. 

\subsection{First variation and regularity of solutions} 
For  $u\in W^{1,2\al}(\Sigma,\Sp^n)$, $V\in W^{1,2\al}(\Sigma, \R^d)$ so that $V(x)\in T_{u(x)}N$ almost everywhere we have 
\begin{equation}\label{eq:1st}
	\ed E_\al (u)[V] = \al \int_\Sigma (2+|\D u|^2)^{\al-1} \D u \cdot \D V \id \Sigma	
\end{equation}
meaning that $\al$-harmonic maps are weak solutions to $\Div_\Sigma ((2+|\D u|^2)^{\al-1} \D u)^\top = 0$ where `$^\top$' is projection onto $T_{u(x)} N$. Since $\al$-harmonic maps are always smooth (\cite[Proposition 2.3]{SU81}) in particular they are solutions to
$$-\Dl u = -A(u)(\D u, \D u) + (\al-1)\frac{\D |\D u|^2\cdot \D u}{2+|\D u|^2}.$$ 

Let $\DD_1\In \R^2$ be a unit Euclidean ball equipped with a metric $g=fg_0$ and $g_0$ is the standard Euclidean metric. If $u\in W^{1,2\al}(\DD,N)$ is $\al$-harmonic then (with all operators below now being Euclidean)  
\beqs 
-\Dl u = -A(u)(\D u, \D u) + (\al-1)\frac{\D (f^{-1}|\D u|^2)\cdot \D u}{2+f^{-1}|\D u|^2}.
\eeqs

We first recall the standard $\eps$-regularity estimate for systems like the above when $\al>1$. For convenience we introduce an extra constant $\nu$, reflecting that we will want to apply these estimates for a re-scaled version of an $\al$-harmonic map, $\hat{u}(x)=u(\la^{-1}x)$, however this has no consequence on the result which can be easily adapted from \cite[Section 3]{SU81}:    

\begin{theorem}\label{thm:reg}
Let $u\in W^{1,2\al}(\DD,N)$, $\nu> 0$, $f\in C^{\infty}(\ov\DD,\R_{>0})$ and $N\emb \R^d$ a smooth closed Riemannian manifold isometrically embedded in some Euclidean space. 
	
There exists $\eps_0 = \eps_0 (N)>0$, $\al_0 = \al_0>1$ and $K(N,\|\log f\|_{C^1})<\infty$ so that for all $1< \al\leq \al_0$ if $u$ weakly solves
\beq\label{eq:aleqs} 
-\Dl u = -A(u)(\D u, \D u) + (\al-1)\frac{\D (f^{-1}|\D u|^2)\cdot \D u}{2\nu+f^{-1}|\D u|^2}
\eeq 
and $\int_{\DD}|\D u|^2 < \eps_0$ then 
	$$\|\D^2 u\|_{L^4(\DD_{\frac12})}+\|\D u\|_{L^\infty(\DD_{\frac12})}\leq K \left(\int_{\DD}|\D u|^2\right)^\fr12.$$
	
\end{theorem}

In particular the above implies:  
\begin{corollary}\label{cor:reg}
Let $\Sigma$ be as in Definition \ref{defSig} and $N\emb \R^d$ a smooth closed Riemannian manifold isometrically embedded in some Euclidean space. 
	
	Let $r<{\rm inj}(\Sigma)$. There exists $\eps_0 = \eps_0 (N)>0$, $\al_0>1$ and $K(N)<\infty$ so that if $u:\Sigma \to N$ is $\al$-harmonic, $1<\al<\al_0$ and $\int_{B_r(a)}|\D u|^2 < \eps_0$ then 
	$$\|\D u\|_{L^{\infty}(B_{\frac{r}{2}}(a))}\leq Kr^{-1}\left(\int_{B_r(a)}|\D u|^2\right)^\fr12.$$ 
\end{corollary}

\subsection{Second variation}

Given $u\in W^{1,2\al}(\Sigma, N)$ and perturbations $V,W\in W^{1,2\al}(\Sigma, \R^d)$ so that $V(x), W(x)\in T_{u(x)}N$ almost everywhere we consider $u_{s,t}=P(u+sV + tW)$ and do a Taylor expansion about $s,t=0$ using \eqref{B1}-\eqref{B2}: $u_{s,t}=u + tW + sV + st A(u)(V,W) + O(t^2) + O(s^2).$
Thus 
\begin{eqnarray*}
 \ed^2 E_\al(u)[V,W]&=&\ppf{^2}{s}{t}{s,t=0}E_\al(u_{s,t})    \\
 &=& \al \int_{\Sigma} (2+|\D u|^2)^{\al-1} \left( \D V \cdot \D W + \D u \cdot \D (A(u)(V,W))\right) \\
 && + 2\al(\al-1) \int_{\Sigma} (2+|\D u|^2)^{\al -2} (\D u \cdot \D V)(\D u \cdot \D W).
 \end{eqnarray*}
Using $\D u \cdot A(u)(V,W) = 0$ we see $\D u \cdot \D(A(u)(V,W)) = \Div( \D u \cdot A(V,W)) - \Dl u \cdot A(u)(V,W)$ and since we have $(\Dl u)^{\bot} = A(u)(\D u, \D u)$ weakly, we have 
\begin{eqnarray*}
\ed^2 E_\al(u)[V,W]&=&\al \int_{\Sigma} (2+|\D u|^2)^{\al-1} \left( \D V \cdot \D W - A(u)(\D u, \D u) \cdot A(u)(V,W)\right) \\
&& + 2\al(\al-1) \int_{\Sigma} (2+|\D u|^2)^{\al -2} (\D u\cdot \D V)(\D u \cdot \D W).
 \end{eqnarray*}
 Of course when $N=\Sp^n$ is a round sphere and $u\in W^{1,2\al}(\Sigma,\Sp^n)$, $V,W\in W^{1,2\al}(\Sigma, \R^d)$ so that $V(x), W(x)\in T_{z(x)}N$ almost everywhere: 
 \begin{eqnarray}\label{eq:2ndvar}
\ed^2 E_\al(u)[V,W]&=&\al \int_{\Sigma} (2+|\D u|^2)^{\al-1} \left( \D V \cdot \D W - |\D u|^2 (V\cdot W)\right) \nn \\
&& + 2\al(\al-1) \int_{\Sigma} (2+|\D u|^2)^{\al -2} (\D u \cdot \D V)(\D u \cdot \D W).
 \end{eqnarray}
 

\bibliographystyle{plain}

\vspace{.5cm}

\begin{flushleft}

  B. Sharp: School of Mathematics, University of Leeds, Leeds LS2 9JT, UK \\\textit{b.g.sharp@leeds.ac.uk} 
\end{flushleft}

\end{document}